\documentclass[12pt,a4paper,reqno]{amsart}

\usepackage{xcolor}
\usepackage[all,2cell,arrow,color]{xy}

\usepackage{amssymb}
\usepackage{mathtools}
\usepackage{tensor}
\numberwithin{equation}{section}
\usepackage{graphicx} 
\usepackage{rotating}
\usepackage{pdflscape}
\definecolor{link}{rgb}{0,0.2,0.8}
\usepackage[pagebackref,colorlinks,citecolor=link,linkcolor=link]{hyperref}
\usepackage{enumerate,xspace}

\UseAllTwocells
\CompileMatrices

\renewcommand{\phi}{\varphi}
\renewcommand{\epsilon}{\varepsilon}
\newcommand{\nto}{\nrightarrow}

 \definecolor{lightgrey}{rgb}{0.666666,0.666666,0.666666}

\newcommand{\cc}{\ensuremath{\mathcal C}\xspace}
\newcommand{\cm}{\ensuremath{\mathcal M}\xspace}
\newcommand{\mM}{\ensuremath{\mathbb M}\xspace}
\newcommand{\cq}{\ensuremath{\mathcal Q}\xspace}

\newcommand{\x}{\times}

\newcommand{\op}{\ensuremath{{}^{\mathsf{op}}}\xspace}

\usepackage{tikz}
\usetikzlibrary{intersections}
\tikzstyle{braid}=[thick]
\tikzset{arr/.style={circle,draw,inner sep=0}}
\tikzset{empty/.style={inner sep=0pt, minimum size=0pt}}

  \newtheorem{proposition}[subsection]{Proposition}
  \newtheorem{lemma}[subsection]{Lemma}
  
  \newtheorem{theorem}[subsection]{Theorem}

  \theoremstyle{definition}
  \newtheorem{definition}[subsection]{Definition}
  \newtheorem{example}[subsection]{Example}

\theoremstyle{remark}
  \newtheorem{remark}[subsection]{Remark}

  \newcounter{c}
  
  \newcommand{\etyk}[1]{\vspace{-7.4mm}$$\begin{equation}\Label{#1}
  \addtocounter{c}{1}}
  \renewcommand{\]}{\ifnum \value{c}=1 $$\else \end{equation}\fi}
\begin{document}

\title{A category of multiplier bimonoids}

\author{Gabriella B\"ohm} 
\address{Wigner Research Centre for Physics, H-1525 Budapest 114,
P.O.B.\ 49, Hungary}
\email{bohm.gabriella@wigner.mta.hu}
\author{Stephen Lack}
\address{Department of Mathematics, Macquarie University NSW 2109, Australia}
\email{steve.lack@mq.edu.au}


\begin{abstract}
The central object studied in this paper is a {\em multiplier bimonoid in a
braided monoidal category $\cc$}, introduced and studied in
\cite{BohmLack:braided_mba}. Adapting the philosophy in
\cite{JanssenVercruysse:mba&mha}, and making some mild assumptions on the
category $\cc$, we consider a category $\cm$ whose objects are certain
semigroups in $\cc$ and whose morphisms $A\to B$ can be regarded as suitable
multiplicative morphisms from $A$ to the multiplier monoid of $B$. We 
equip this category $\cm$ with a monoidal structure and describe multiplier
bimonoids in $\cc$ (whose structure morphisms belong to a distinguished class
of regular epimorphisms) as certain comonoids in $\cm$. This provides us with
one possible notion of morphism between such multiplier bimonoids.
\end{abstract}
  
\maketitle


\section{Introduction}

A bialgebra over a field or, more generally, a bimonoid in a braided monoidal
category, is an object carrying a monoid and a comonoid structure subject to
compatibility conditions that can be interpreted as saying that a bimonoid is
a monoid in the category of comonoids; equivalently, it is a comonoid in the
category of monoids.

A multiplier bialgebra over a field
\cite{BohmGomezTorrecillasLopezCentella:wmba} or, more generally, a multiplier
bimonoid in a braided monoidal category \cite{BohmLack:braided_mba}, is a
generalization which is no longer a monoid or a comonoid in the base
category. However, Janssen and Vercruysse constructed in
\cite{JanssenVercruysse:mba&mha} a monoidal category, whose objects are
certain non-unital algebras (say over a field), and in which the comonoids
include the multiplier Hopf algebras of
Van~Daele~\cite{VanDaele:multiplier_Hopf}.  
 
Our aim in this paper is to generalize and strengthen this result. Namely,
under mild assumptions (involving a class \cq of regular epimorphisms) we
construct a category $\cm$ of certain semigroups in a braided monoidal
category $\cc$. We describe multiplier bimonoids in $\cc$ (whose structure
morphisms lie in \cq) as certain comonoids in $\cm$. Defining the morphisms
betweeen such multiplier bimonoids as the morphisms between the corresponding
comonoids in $\cm$, we obtain a category of multiplier bimonoids in $\cc$. 

{\bf Acknowledgement.} We gratefully acknowledge the financial support of the
Hungarian Scientific Research Fund OTKA (grant K108384) and the Australian
Research Council Discovery Grant (DP130101969), as well as an ARC Future Fellowship
(FT110100385). The second-named author is grateful for the warm hospitality of his hosts during visits to the Wigner Research Centre in Sept-Oct 2014 and Aug-Sept 2015.


\section{Multiplier monoids and their morphisms}

We begin by describing what is meant by multiplier monoids in closed
braided monoidal categories and we  characterize multiplicative morphisms with
codomain a multiplier monoid.  

\subsection{Assumptions on the base category}
\label{sect:assumptions}
Throughout, we work in a braided monoidal category $\cc$. The composition of
morphisms $f\colon A\to B$ and $g\colon B\to C$ will be denoted by $g.f\colon
A\to C$. 
The monoidal product of $A$ and $B$ will be denoted by $AB$, the monoidal unit
by $I$ and the braiding by $c$. For $n$ copies of the same object $A$, we also
use the power notation $AA\dots A=A^n$.  

We fix a class \cq of regular epimorphisms in \cc which is closed under
composition and monoidal product, contains the isomorphisms, and is
right-cancellative in the sense that if $s\colon A\to B$ and $t.s\colon A\to
C$ are in \cq, then so is $t\colon B\to C$. Since each $q\in\cq$ is 
a regular epimorphism, it is the coequalizer of some pair of maps. Finally we
suppose that this pair may be chosen in such a way that the coequalizer is
preserved by taking the monoidal product with any object. 

These assumptions are always satisfied when \cq consists of the split
epimorphisms. In the other main case \cq consists of the regular
epimorphisms. In this case, we need to suppose that the regular epimorphisms
are closed under composition, as is the case in any regular category; we also
need to suppose that (enough) coequalizers are preserved by taking the
monoidal product with a fixed object, as will be true if the monoidal category
is {\em closed} (see Paragraph~\ref{claim:closed} below). In particular, we
may take \cq to be all the regular epimorphisms if \cc is the symmetric
monoidal category of modules over a commutative ring.   

\subsection{Semigroups with non-degenerate multiplication.}
By a {\em semigroup} in the braided monoidal category $\cc$ we mean a pair
$(A,m)$ consisting of an object $A$ of $\cc$ and a morphism $m:A^2\to A$ --
called the {\em  multiplication} -- which obeys the associativity condition
$m.m1=m.1m$. If the semigroup has a {\em unit} -- that is, a morphism $u\colon
I\to A$ such that $m.u1=1=m.1u$ -- then we say that $A$ is a {\em monoid}. 

The multiplication -- or, alternatively, the semigroup $A$ -- is said to be {\em
non-degenerate} if for any objects $X,Y$ of $\cc$, both maps
\begin{eqnarray*}
&\cc(X,YA) \to \cc(XA,YA),\qquad
&f\mapsto \xymatrix{XA\ar[r]^-{f1}& YA^2\ar[r]^-{1m} &YA}\quad \textrm{and}\\
&\cc(X,AY) \to \cc(AX,AY),\qquad
&g\mapsto \xymatrix{AX\ar[r]^-{1g}& A^2Y\ar[r]^-{m1} &AY}
\end{eqnarray*}
are injective. The multiplication of a monoid is always
non-degenerate. (Requiring injectivity of these maps for {\em any} object $Y$ 
is quite strong and can often be avoided. For a careful analysis of the class
of objects $Y$ with this property consult \cite{BohmLackGomTor}.) 

\subsection{Closed monoidal categories.}
\label{claim:closed}
The braided monoidal category \cc is said to be {\em closed} if for each 
object $X$ the functor $X(-)\colon\cc\to\cc$ possesses a right adjoint
(equivalently, if each $(-)X$ possesses a right adjoint). We write $[X,-]$
for the right adjoint; then the components of the unit and the counit have the
form  
$$
\xymatrix{
Y \ar[r]^-{\eta} & [X,XY], & X [X,Y] \ar[r]^-{\epsilon} & Y. }
$$
The right adjoint $[X,-]$
is functorial in the variable $X$, so that in fact we have a functor
$[-,-]\colon \cc\op \x\cc\to\cc$ and now the adjointness 
$$\cc(XY,Z) \cong \cc(Y,[X,Z])$$
is natural in all three variables.

\begin{lemma}\label{lem:non-deg}
For a semigroup $A$ in a closed braided monoidal category $\cc$, the following
assertions are equivalent.

(i) The multiplication $m\colon A^2\to A$ is non-degenerate.

(ii) For any object $Y$, both morphisms 
\begin{eqnarray*}
r_Y&:=&\xymatrix{
AY \ar[r]^-\eta &
[A,A^2Y] \ar[r]^-{[A,m1]} &
[A,AY]}
\quad \textrm{and} \\
l_Y&:=&\xymatrix{
YA \ar[r]^-\eta &
[A,AYA] \ar[r]^-{[A,c1]} &
[A,YA^2] \ar[r]^-{[A,1c]} &
[A,YA^2] \ar[r]^-{[A,1m]} &
[A,YA]}
\end{eqnarray*}
are monomorphisms. 
\end{lemma}

\begin{proof}
Since $r_Y$ and $m1\colon A^2Y\to AY$ are mates under the adjunction $A(-)\dashv
[A,-]$, the equality $r_Y.f=r_Y.g$ holds for any morphisms $f$ and
$g\colon X\to AY$ if and only if $m1.1f=m1.1g$.
Symmetrically, since $l_Y$ and $1m.c_{A,YA}\colon AYA\to YA$ are mates, the equality
$l_Y.f=l_Y.g$ holds for any $f,g:X\to YA$ if and only if $1m.f1=1m.g1$.
\end{proof}

\subsection{$\mM$-morphisms}\label{sect:mM}
For a monoid $B=(B,m,u)$ and an object $A$, to give a morphism $f\colon A\to
B$ in \cc is equivalently to give a morphism $f_1\colon AB\to B$ compatible
with the right actions of $B$, in the sense that the first diagram in
\eqref{eq:multiplier} below commutes 
\begin{equation}\label{eq:multiplier}
\xymatrix{ 
AB^2 \ar[r]^{1m} \ar[d]_{f_11} & AB \ar[d]^{f_1} \\
B^2 \ar[r]_{m} & B}
\qquad
\xymatrix{
B^2A \ar[r]^{m1} \ar[d]_{1f_2} & BA \ar[d]^{f_2} \\
B^2 \ar[r]_{m} & B. }
\end{equation}
Under this bijection $f_1$ is given by $m.f1$, and $f$ given by
$f_1.1u$. Dually, it is equivalent to giving a morphism $f_2\colon BA \to B$
compatible with the left actions, in the sense that the second diagram in
\eqref{eq:multiplier} commutes. Furthermore, the resulting $f_1$ and $f_2$ are
related by commutativity of the diagram   
\begin{equation}\label{eq:component-compatibility}
\xymatrix{
BAB\ar[r]^-{1f_1} \ar[d]_-{f_2 1} &
B^2\ar[d]^-m\\
B^2\ar[r]_-m &
B .}
\end{equation}

If now $A$ is a semigroup, then $f\colon A\to B$ will be a semigroup morphism
if and only if the diagrams 
\begin{equation}\label{eq:multiplicative}
\xymatrix{
A^2B \ar[r]^-{1f_1} \ar[d]_-{m1} &
AB \ar[d]^-{f_1}
&&
BA^2 \ar[r]^-{f_21} \ar[d]_-{1m} &
BA \ar[d]^-{f_2} \\
AB \ar[r]_-{f_1} &
B 
&&
BA \ar[r]_-{f_2} &
B}
\end{equation}
commute. 

This motivates the following notion of morphism when $B$ is just a
non-degenerate semigroup. 

\begin{definition}
If $A$ is an object and $B$ is a non-degenerate semigroup in \cc, an {\em
  $\mM$-morphism} $f$ from $A$ to $B$ is a pair $(f_1,f_2)$ of morphisms in
\cc making the diagram \eqref{eq:component-compatibility} commute. We call
$f_1$ and $f_2$ the {\em components} of $f$, and we represent the
$\mM$-morphism as $f\colon A\nrightarrow B$. If $A$ is also a semigroup, we
say that the $\mM$-morphism $f$ is {\em multiplicative} if the diagrams
\eqref{eq:multiplicative} commute.  
\end{definition}

\begin{remark}~\label{rem:components_equivalent}
By non-degeneracy, for $\mM$-morphisms $f$ and $g$ to be equal, it suffices
that either $f_1=g_1$ or $f_2=g_2$; the other equality then
follows. Similarly, for an $\mM$-morphism to be multiplicative it suffices
that either of the diagrams in \eqref{eq:multiplicative} commutes;
commutativity of the other then follows.  
\end{remark}

\begin{lemma}\label{lem:module-maps}
If $B$ is a non-degenerate semigroup, and $f\colon A\nrightarrow B$ is an
$\mM$-morphism, then the diagrams in \eqref{eq:multiplier} commute.  
\end{lemma}

\begin{proof}
Commutativity of the first diagram in the claim follows by commutativity of
both diagrams 
$$
\xymatrix{
BAB^2\ar[r]^-{1f_11}\ar[d]_-{f_211} 
\ar@{}[rd]|-{\eqref{eq:component-compatibility}}&
B^3 \ar[rr]^-{1m} \ar[d]_-{m1} 
\ar@{}[rrd]|-{\textrm{(associativity)}}&&
B^2\ar[d]^-m
&
BAB^2 \ar[r]^-{11m}\ar[d]_-{f_211} &
BAB \ar[r]^-{1f_1}\ar[d]_-{f_21} 
\ar@{}[rd]|-{\eqref{eq:component-compatibility}}&
B^2\ar[d]^-m\\
B^3\ar[r]_-{m1} &
B^2 \ar[rr]_-m &&
B
&
B^3\ar[r]_-{1m}&
B^2\ar[r]_-m &
B}
$$
and the associativity and non-degeneracy of $m$.
A symmetric reasoning applies to the second diagram.
\end{proof}

\subsection{Multiplier monoids}\label{sect:MA}
Suppose that $X$ is an object of \cc and $A$ is a non-degenerate semigroup. If
\cc is closed, then there is a bijection between morphisms $f_1\colon XA\to A$
and morphisms $\check{f}_1\colon X\to [A,A]$, and similarly there is a
bijection between morphisms $f_2\colon AX\to A$ and morphisms $\hat{f}_2\colon
X\to [A,A]$. Moreover, the morphisms $f_1$ and $f_2$ make the diagram
\eqref{eq:component-compatibility} commute, and so determine an
$\mM$-morphism, just when $\check{f}_1$ and $\hat{f}_2$ make the first diagram
in 
\begin{equation}\label{eq:M(A)}
  \xymatrix{ X \ar[r]^{\check{f}_1} \ar[d]_{\hat{f}_2} & [A,A] \ar[d]^\phi \\
[A,A] \ar[r]_\psi & [A^2,A] }
\qquad
\xymatrix{
\mM(A)\ar[r]^{\check{e}_1} \ar[d]_{\hat{e}_2}  &
[A,A]\ar[d]^\phi \\
[A,A]\ar[r]_\psi &
[A^2,A].}
\end{equation}
commute, where $\phi$ and $\psi$ correspond under 
the adjunction isomorphism $\cc(A^2[A,A],A)$ $\cong \cc([A,A],[A^2,A])$ to the
morphisms  
$$
\xymatrix{A^2[A,A]\ar[r]^-{1\epsilon} & A^2\ar[r]^-m & A,}
\qquad
\xymatrix{A^2[A,A]\ar[r]^-{1c}&A[A,A]A \ar[r]^-{\epsilon 1} & A^2 \ar[r]^-m &A.}
$$
If the pullback of $\phi$ and $\psi$ exists, as in the second 
diagram of \eqref{eq:M(A)}, then there is a bijection between $\mM$-morphisms
$X\nrightarrow A$, and morphisms $X\to \mM(A)$ in \cc.  

Under this bijection, the identity morphism $\mM(A)\to \mM(A)$ will correspond
to an $\mM$-morphism $e\colon \mM(A)\nto A$ with components $e_1\colon
\mM(A)A\to A$ and $e_2\colon A\mM(A)\to A$. 
The components of a morphism $f\colon X\to \mM(A)$ have the form $f_1=e_1.f1$
and $f_2=e_2.1f$. 

\begin{proposition}\label{prop:M(A)-monoid}
Consider a non-degenerate semigroup  $A$ in a closed braided monoidal category
$\cc$.  
\begin{enumerate}[(i)]
\item If the pullback $\mM(A)$ in \eqref{eq:M(A)} exists, then it carries
  the structure of a unital monoid in \cc.
\item For another semigroup $B$, a morphism $f:B\to \mM(A)$ in \cc is
  multiplicative if and only if the the corresponding $\mM$-morphism $B\nto A$
  is so. 
\end{enumerate}
\end{proposition}

\begin{proof}
(i) Using \eqref{eq:component-compatibility} for $e$ and functoriality of the
  monoidal product, we see 
that $e_1.1e_1:\mM(A)^2A\to A$ and $e_2.e_2 1:A\mM(A)^2 \to A$ can be regarded
as the components of a morphism $m:\mM(A)^2\to \mM(A)$ rendering commutative
\begin{equation}\label{eq:M(A)-product}
\xymatrix{
\mM(A)^2 A\ar[r]^-{1e_1} \ar[d]_-{m1}&
\mM(A) A \ar[r]^-{e_1} &
A \ar@{=}[d]&
A\mM(A) \ar[l]_-{e_2} &
A\mM(A)^2 \ar[l]_-{e_2 1} \ar[d]^-{1m}\\
\mM(A)A\ar[rr]_-{e_1}&&
A &&
A\mM(A) .\ar[ll]^-{e_2}}
\end{equation}
Applying \eqref{eq:M(A)-product} and functoriality of the monoidal product,
the components of $m.1m$ and of $m.m1$ turn out to be equal to the same
morphisms $e_1.1e_1.11e_1$ and $e_2.e_2 1.e_2 11$. This proves the
associativity of $m$.

The identity morphism $1:A\to A$ can be regarded as the first and the second
components of a morphism $u:I\to \mM(A)$ rendering commutative
\begin{equation}\label{eq:M(A)-unit}
\xymatrix{
&
A \ar@{=}[d]\ar[ld]_-{u1}\ar[rd]^-{1u} \\
\mM(A)A\ar[r]_-{e_1}&
A&
A\mM(A).\ar[l]^-{e_2}}
\end{equation}
By \eqref{eq:M(A)-product}, \eqref{eq:M(A)-unit} and functoriality of the
monoidal product, the components of both $m.1u$ and of $m.u1$ are equal to
$e_1$ and $e_2$. This proves that $u$ is the unit of $m$.

(ii) By \eqref{eq:M(A)-product}, the components of
$m.ff:B^2\to \mM(A)$ are $f_1.1f_1$ and $f_2.f_21$; while the components of
$f.m:B^2\to \mM(A)$ are $f_1.m1$ and $f_2.1m$. 
\end{proof}


\section{A category of semigroups}

In the previous section we introduced a notion of $\mM$-morphism for
non-degenerate semigroups; we now turn to composition of $\mM$-morphisms. This
does not seem to be possible in general, but we give sufficient conditions
under which it is.  Once again, we motivate the definition using the unital
case. If $g\colon B\to C$ is a monoid morphism, and $f\colon A\to B$ an
arbitrary morphism,  then the following diagram commutes 
$$\xymatrix{
ABC \ar[r]^{1g1} \ar[d]_{f11} & 
AC^2 \ar[r]^{1m} & 
AC \ar[d]^{f1} \\
B^2C \ar[r]^{1g1} \ar[dd]_{m1} & 
BC^2 \ar[r]^{1m} \ar[d]^{g11} & 
BC \ar[d]^{g1} \\
& C^3 \ar[r]^{1m} \ar[d]^{m1} & 
C^2 \ar[d]^{m} \\
BC \ar[r]_{g1} & 
C^2 \ar[r]_{m} & 
C }$$
which can in turn be read as the equality $(gf)_1.1g_1=g_1.f_11$ using the
notation of Paragraph~\ref{sect:mM}. 

Now suppose that $f\colon A\nto B$ and $g\colon B\nto C$ are $\mM$-morphisms
with $g$ multiplicative. We would like to define a composite $\mM$-morphism
$g\bullet f$ of $g$ and $f$ in such a way that the diagrams 
\begin{equation}
  \label{eq:c}
\xymatrix{
ABC \ar[r]^{1g_1} \ar[d]_{f_11} & AC \ar[d]^{(g\bullet f)_1} \\
BC \ar[r]_{g_1} & C }
\qquad
\xymatrix{
CA \ar[d]_{(g\bullet f)_2} & CBA \ar[l]_{g_21} \ar[d]^{1f_2} \\
C & CB \ar[l]^{g_2} }
\end{equation}
commute. 

For a general $\mM$-morphism $g$, there might be many $g\bullet f$ making
these diagrams commute, but if the maps $1g_1$ and $g_21$ are epimorphisms,
there can be at most one. As far as the existence of $g\bullet f$, this will
clearly become easier to analyzie if $1g_1$ and $g_21$ are regular
epimorphisms. In fact it will turn out that there is a $g\bullet f$ provided
that $g_1$ and $g_2$ lie in \cq, in which case we say that the $\mM$-morphism
$g$ is {\em dense}. The key step is the following result.

\begin{lemma}\label{lem:composition}
Let $f\colon A\nto B$ and $g\colon B\nto C$ be $\mM$-morphisms with $g$ dense
and multiplicative; in particular, this includes non-degeneracy of $C$. Then
for any morphism $s:X\to BC$, the composite
$$
\xymatrix{
AX\ar[r]^-{1s} &
ABC \ar[r]^-{f_11} &
BC \ar[r]^-{g_1} &
C}
$$
depends on $s$ only through $g_1.s$. Dually, for any morphism $s:X\to CB$, the
composite 
$$
\xymatrix{
XA\ar[r]^-{s1} &
CBA \ar[r]^-{1f_2} &
CB \ar[r]^-{g_2} &
C}
$$
depends on $s$ only through $g_2.s$.
\end{lemma}

\begin{proof}
The equal paths around
$$
\xymatrix@C=35pt@R=15pt{
BAX\ar[r]^-{11s} \ar[dd]_-{f_21} &
BABC \ar[rr]^-{1f_11} \ar[dd]^-{f_211} 
\ar@{}[rrd]|-{\eqref{eq:component-compatibility}} &&
B^2C\ar[rr]^-{1g_1} \ar[d]_-{m1} 
\ar@{}[rrd]|-{\eqref{eq:multiplicative}} &&
BC \ar[dd]^-{g_1}\\
&&&
BC \ar[rrd]^-{g_1}  
\ar@{}[d]|(.6){\eqref{eq:multiplicative}} &&\\
BX \ar[r]_-{1s} &
B^2C\ar[rru]^-{m1} \ar[rr]_-{1g_1} &&
BC \ar[rr]_-{g_1} &&
C}
$$
clearly depend only on $g_1.s$. Thus the common composite
$$
\xymatrix{
CBAX \ar[r]^-{111s} \ar[d]_-{g_211} &
CBABC \ar[r]^-{11f_11} &
CB^2C \ar[r]^-{11g_1} &
CBC \ar[r]^-{1g_1} \ar[d]_-{g_21} 
\ar@{}[rd]|-{\eqref{eq:component-compatibility}} &
C^2 \ar[d]^-m \\
CAX \ar[r]_-{11s} &
CABC \ar[r]_-{1f_11} &
CBC \ar[r]_-{1g_1} &
C^2 \ar[r]_-m &
C}
$$
depends only on $g_1.s$. Since $g_2$ belongs to \cq so does 
$g_211$, and thus the bottom row of this last diagram depends only on
$g_1.s$. Finally by non-degeneracy of the multiplication of $C$ we conclude
the first claim. The other claim follows symmetrically. 
\end{proof}

\begin{proposition}\label{prop:composition}
If $f\colon A\nto B$ is an $\mM$-morphism and $g\colon B\nto C$ is a dense
multiplicative $\mM$-morphism then there is a unique $\mM$-morphism $g\bullet
f\colon A\nto C$ making the diagrams \eqref{eq:c} commute.  
Furthermore, $g\bullet f$ is dense or multiplicative if $f$ is so.
\end{proposition}

\proof 
First, since $g$ is dense, $g_1\colon BC\to C$ is the coequalizer of maps
$s,s'\colon X\to BC$, and this coequalizer is preserved by $A(-)$, so that
also $1g_1\colon ABC\to AC$ is the coequalizer of $1s$ and $1s'$. By 
Lemma~\ref{lem:composition}, the composites $g_1.f_11.1s$ and $g_1.f_11.1s'$
are equal, and so there is a unique map $(g\bullet f)_1$ making the diagram in
\eqref{eq:c} commute; similarly there is a unique induced $(g\bullet f)_2$.  

Next we show that $(g\bullet f)_1$ and $(g\bullet f)_2$ are the components of
an $\mM$-morphism. To do so, observe that in the commutative diagrams  
$$
\xymatrix@C=15pt{
CBABC \ar[rr]^-{g_2111} \ar[dr]^{11f_1 1} \ar[d]_{1f_211} &&
CABC \ar[r]^-{11g_1} \ar[d]_-{1f_11} \ar@{}[rd]|-{\eqref{eq:c}} &
CAC \ar[d]^{1(g\bullet f)_1} \\
CB^2C \ar[d]_{1m1} \ar@{}[r]|-{\eqref{eq:component-compatibility}} &
CB^2C \ar[r]^-{g_211} \ar[ld]^-{1m1} \ar@{}[d]^-{~\eqref{eq:multiplicative}}&
CBC \ar[r]^{1g_1} \ar[d]_-{g_21} 
\ar@{}[rd]|-{\eqref{eq:component-compatibility}} & 
C^2 \ar[d]^{m} \\
CBC \ar[rr]_{g_21} &&
C^2 \ar[r]_{m} & 
C}
\xymatrix{
CBABC \ar[r]^-{111g_1} \ar[d]_-{1f_211} &
CBAC \ar[r]^-{g_211} \ar[d]_-{1f_21} \ar@{}[rd]|-{\eqref{eq:c}}&
CAC \ar[d]^{(g\bullet f)_21} \\
CB^2C \ar[r]^{11g_1} \ar[d]_{1m1} \ar@{}[dr]|-{~\eqref{eq:multiplicative}} &
CBC \ar[d]_-{1g_1} \ar[r]^{g_21} 
\ar@{}[rd]|-{\eqref{eq:component-compatibility}} & 
C^2 \ar[d]^-{m} \\
CBC \ar[r]_{1g_1} & 
C^2 \ar[r]_{m} & 
C}
$$
the bottom rows are equal by $\eqref{eq:component-compatibility}$, and so the
top-right paths are equal. But $g_21g_1$ in the top rows is an epimorphism
since $g$ is dense, so that   
the right verticals are equal as required.  

If $f$ is dense, then $f_11$, $g_1$, and $1g_1$ are all in \cq, hence so too
is $(g\bullet f)_1$; similarly $(g\bullet f)_2$ is in \cq and so $g\bullet f$
is dense.  

Finally we show that $g\bullet f$ is multiplicative if $f$ is so. In the
commutative diagrams
$$\xymatrix{
A^2BC \ar[r]^{11g_1} \ar[dr]_{1f_11} \ar[d]_{m11} & 
A^2C \ar[dr]^{1(g\bullet f)_1} \ar@{}[d]|-{\eqref{eq:c}} \\
ABC \ar[dr]_{f_11} \ar@{}[r]|-{~\eqref{eq:multiplicative}} & 
ABC \ar[d]^{f_11} \ar[r]^{1g_1} \ar@{}[dr]|-{~\eqref{eq:c}} & 
AC \ar[d]^{(g\bullet f)_1} \\
&  BC \ar[r]_{g_1} & C }
\xymatrix{
A^2BC \ar[r]^{11g_1}  \ar[d]_{m11} & 
A^2C \ar[d]^{m1}  \\
ABC \ar[dr]_{f_11} \ar[r]^{1g_1} & AC \ar[dr]^{(g\bullet f)_1}  
\ar@{}[d]|-{\eqref{eq:c}}  \\
&  BC \ar[r]_{g_1} & C }$$
the map $11g_1$ is an epimorphism since $g$ is dense, and so $(g\bullet
f)_1.1(g\bullet f)_1=(g\bullet f)_1.m1$ as required.  
\endproof

Next we turn to associativity of this composition. 

\begin{proposition}
Let $f\colon A\nto B$, $g\colon B\nto C$, and $h\colon C\nto D$ be
$\mM$-morphisms, and suppose that $g$ and $h$ are dense and
multiplicative. Then $(h\bullet g)\bullet f=h\bullet(g\bullet f)$.   
\end{proposition}

\proof
Since $11h_1\colon ABCD \to ABD$ and $1(h\bullet g)_1\colon ABD \to AD$ are epimorphisms,
this follows immediately from the commutativity of the following diagrams 
$$\xymatrix{
ABCD \ar[r]^{11h_1} \ar[d]_{f_111} & ABD \ar[dr]^{1(h\bullet g)_1} \ar[d]_{f_11} \\
BCD \ar[r]^{1h_1} \ar[dr]_{g_1 1} & BD \ar[dr]_{(h\bullet g)_1}  & 
AD \ar[d]^{((h\bullet g)\bullet f)_1} \\
& CD \ar[r]_{h_1} & D }
\xymatrix{
ABCD \ar[r]^{11h_1} \ar[d]_{f_111} \ar[dr]^{1g_11} & 
ABD \ar[dr]^{1(h\bullet g)_1} \\
BCD \ar[dr]_{g_11} & ACD \ar[r]^{1h_1} \ar[d]^{(g\bullet f)_11} & 
AD \ar[d]^{(h\bullet(g\bullet f))_1} \\
& CD \ar[r]_{h_1} & D }
$$
in which the top left square in the left diagram commutes by functoriality of
the monoidal product, and all remaining regions commute by instances of
\eqref{eq:c}.  
\endproof

As for the identity morphisms, the unital case suggests that the identity
$\mM$-morphism $i$ on a non-degenerate semigroup $A$ should have components
$i_1$ and $i_2$ equal to $m$. This does indeed define a multiplicative
$\mM$-morphism by associativity of $m$; it will be dense just when $m$ lies in
\cq. It follows from the non-degeneracy of $A$ that $i\colon A\to \mM(A)$ is a
monomorphism in $\cc$ preserved by the functor $X(-)$ for any object $X$.  

\begin{proposition}
Let $f\colon A\nto B$ be an $\mM$-morphism.
\begin{enumerate}[(i)]
\item if $f$ is dense and multiplicative, then $f\bullet i=f$;
\item if $i$ is dense, then $i\bullet f=f$.
\end{enumerate}
\end{proposition}

\begin{proof}
Part (i) follows by commutativity of the diagrams in \eqref{eq:multiplicative}
and part (ii) follows by Lemma \ref{lem:module-maps}. 
\end{proof}

In particular, we now have a category.

\begin{proposition}\label{prop:category-M}
There is a category $\cm$, whose objects are the non-degenerate semigroups
with multiplication in \cq, and whose morphisms are the dense multiplicative
$\mM$-morphisms. The composite $g\bullet f$ of composable morphisms $g$ and
$f$ has components as in \eqref{eq:c}, and the identity on an object $A$ is
the $\mM$-morphism $i\colon A\nto A$ with components equal to the
multiplication $m$. The monoidal unit $I$ equipped with the trivial
multiplication is initial in the category \cm.  
\end{proposition}

\proof The only thing that remains to be proven is that $I$ is initial. For
each object $A$, there is an $\mM$-morphism $u\colon I\nto A$ with components
equal to the identity morphism of $A$. This is clearly multiplicative and
dense. A general $\mM$-morphism $v\colon I\nto A$ will have components given
by endomorphisms $v_1,v_2\colon A\to A$ satisfying $m.1v_1=m.v_21$. This will
be multiplicative if and only if $v_1$ and $v_2$ are idempotent, 
and it will be dense if and only if $v_1$ and $v_2$ are in \cq; but the only
epimorphic idempotents are the identities.  
\endproof

\begin{remark} \label{rem:no-braid}
The proof shows that, in fact, Proposition \ref{prop:category-M} holds
also for not necessarily braided monoidal categories $\cc$. 
\end{remark}

The category studied by Janssen and Vercruysse in
\cite{JanssenVercruysse:mba&mha} is the case where \cc consists of all modules
over a commutative ring, but where we only consider projective modules in
defining $\cm$. The construction of \cm is reminiscent of the Kleisli
construction, but does not seem literally to be an example; the obstruction is
the need to restrict to dense morphisms. 

\begin{lemma}\label{lem:i-e}
For a semigroup $A$ with non-degenerate multiplication $m$, the following
diagram commutes. 
$$
\xymatrix{
\mM(A)A\ar[r]^-{1i}\ar[d]_-{e_1} &
\mM(A)^2 \ar[d]^-m &
A \mM(A) \ar[l]_-{i1} \ar[d]^-{e_2} \\
A\ar[r]_-i &
\mM(A) &
A \ar[l]^-i}
$$
\end{lemma}

\begin{proof}
We only prove commutativity of the square on the left, symmetric reasoning
applies to the other one. In view of Remark \ref{rem:components_equivalent},
it is enough to compare the first components of the morphisms around the
square.
$$
\xymatrix{
\mM(A)A^2\ar[r]^-{1i1}\ar[rd]_-{1m} &
\mM(A)^2A\ar[r]^-{m1}\ar[d]^-{1e_1} \ar@{}[rd]|-{\eqref{eq:M(A)-product}}&
\mM(A)A\ar[d]^-{e_1}
&
\mM(A)A^2\ar[r]^-{e_11} \ar[d]_-{1m} \ar@{}[rd]|-{\eqref{eq:multiplier}}&
A^2 \ar[r]^-{i1}\ar[rd]^-m &
\mM(A)A\ar[d]^-{e_1} \\
& \mM(A)A\ar[r]_-{e_1} &
A
&
\mM(A)A\ar[rr]_-{e_1} &&
A}
$$
The triangular regions commute since the first component of $i\colon A\to \mM(A)$ is
the multiplication $m$. 
\end{proof}

\begin{proposition}\label{prop:gtilde}
Consider a closed braided monoidal category $\cc$ and semigroups
$A,B$ in $\cc$. Assume that the multiplication of $B$ is non-degenerate and
that the pullbacks $\mM(A)$ and $\mM(B)$ in \eqref{eq:M(A)} exist. Then for
any dense and multiplicative morphism $g\colon A\to \mM(B)$ there is a unique
monoid morphism $\widetilde g\colon\mM(A)\to \mM(B)$ obeying $\widetilde
g.i=g$.   
\end{proposition}

\begin{proof}
Recall from Paragraph~\ref{sect:MA} the multiplicative $\mM$-morphism $e\colon
\mM(A)\nto A$, and regard $g$ as a dense multiplicative $\mM$-morphism $A\nto
B$. The composite $g\bullet e\colon \mM(A)\nto B$ can in turn be regarded as a
multiplicative moprhism $\tilde{g}\colon \mM(A)\to \mM(B)$.  

Explicitly, the components of $\widetilde{g}$ are determined by commutativity
of the following diagrams.   
\begin{equation}\label{eq:gtilde}
\xymatrix{
\mM(A)AB \ar[r]^-{1g_1} \ar[d]_-{e_1 1} &
\mM(A)B \ar@{-->}[d]^-{\widetilde g_1}
&&
BA\mM(A) \ar[r]^-{g_21}\ar[d]_-{1e_2} &
B\mM(A) \ar@{-->}[d]^-{\widetilde g_2} \\
AB \ar[r]_-{g_1} & 
B
&&
BA \ar[r]_-{g_2} &
B .}
\end{equation}

Since $1g_1$ is an epimorphism, commutativity of 
$$
\xymatrix{
A^2B\ar[d]_-{1g_1}  \ar[dr]^-{i11} \ar@{=}[rrr] &&&
A^2B \ar[ld]_-{m1} \ar[d]^-{1g_1} \\
AB\ar[d]_-{i1} &
\mM(A)AB\ar[ld]^-{1g_1} \ar[r]^-{e_1 1} 
\ar@{}[rd]|-{\eqref{eq:gtilde}}&
AB \ar[rd]_-{g_1} 
\ar@{}[r]|-{\eqref{eq:multiplicative}}&
AB \ar[d]^-{g_1} \\
\mM(A)B\ar[rrr]_-{\widetilde g_1} &&&
B .}
$$
implies that $(\widetilde g.i)_1=\widetilde g_1.i1$ is equal to $g_1$. Hence 
it follows by Remark \ref{rem:components_equivalent} that $\widetilde g.i=g$.

Conversely, suppose  that  $h\colon\mM(A)\to\mM(B)$ is a multiplicative morphism
satisfying $h.i=g$; equivalently, $h_1.i1=g_1$. Then 
$$
\xymatrix{
\mM(A)AB \ar[r]_-{1i1} \ar[d]_-{e_11} \ar@/^1pc/[rr]^-{1g_1}&
\mM(A)^2B \ar[r]_-{1h_1} \ar[d]^-{m1} &
\mM(A)B \ar[d]^-{h_1}\\
AB \ar[r]^-{i1} \ar@/_1pc/[rr]_-{g_1}&
\mM(A)B \ar[r]^-{h_1} &
B}
$$
commutes by Lemma \ref{lem:i-e} and  Proposition
\ref{prop:M(A)-monoid}~(ii). By the uniqueness of $\widetilde g_1$ rendering
commutative the first diagram in \eqref{eq:gtilde} we conclude that
$h_1=\widetilde g_1$ and thus by Remark \ref{rem:components_equivalent}
also $h=\widetilde g$. 

It remains to see that $\widetilde g$ is unital, or equivalently that
$g\bullet u=u$; but this follows from the fact that $I$ is initial in \cm.  
\end{proof}

\begin{remark} \label{rem:A->B}
We motivated the definition of $\mM$-morphisms $f\colon A\nto B$ by the fact
that any monoid morphism $f\colon A\to B$ induces such an $\mM$-morphism with
components  
\begin{equation} \label{eq:A->B} 
\xymatrix{
AB \ar[r]^-{f1} &
B^2 \ar[r]^-m &
B
&&
BA \ar[r]^-{1f} &
B^2 \ar[r]^-m &
B}
\end{equation}
which render commutative the diagrams in \eqref{eq:component-compatibility} and
in \eqref{eq:multiplicative}. But if $A$ and $B$ are merely semigroups and
$f\colon A\to B$ multiplicative, then the same definitions still give a
multiplicative $\mM$-morphism $A\nto B$, which we call $f^\#$; it is just that
in this non-unital case the two notions are no longer equivalent. If $g\colon
B\nto C$ is a morphism in \cm and $f\colon A\to B$ a morphism in \cc, then 
the composite $g\bullet f^\#$ has components 
$(g\bullet f^\#)_1=g_1.f1$ and 
$(g\bullet f^\#)_2=g_2.1f$. 
On the other hand, if the multiplication of $B$ is non-degenerate and it
belongs to \cq, and $f\colon A\to B$ is a multiplicative {\em isomorphism} in
$\cc$ and $g\colon Z \nto A$ an arbitrary $\mM$-morphism then $(f^\#\bullet
g)_1=f.g_1.1f^{-1}$ and $( f^\# \bullet g)_2=f.g_2.f^{-1}1$. 
\end{remark}

We record various facts about the passage from $f$ to $f^\#$ in the
following proposition. 

\begin{proposition}\label{prop:f-sharp}
There is a non-full subcategory $\mathcal D$ of the category of semigroups in
\cc whose objects are those semigroups which are non-degenerate and have
multiplication in \cq, and whose morphisms $f\colon A\to B$ are those
semigroup morphisms for which the induced $f^\#$ have components lying in
\cq. There is a faithful functor $\mathcal D \to\cm$ which is the 
identity on objects and sends $f$ to $f^\#$; furthermore, this functor is full
on isomorphisms.  
\end{proposition}

\proof
The existence of $\mathcal D$ and the faithful functor is evident from the
previous discussion. We shall therefore only verify the fact that the functor
is full on isomorphisms.  

Suppose then that $f\colon A\nto B$ is an isomorphism in $\cm$, say with
inverse $g$. The components $g_1$ and $g_2$ of $g$ lie in \cq, thus in
particular $g_1$ is the coequalizer of a pair $w,v\colon X\to BA$ of morphisms
in \cc. In the diagram \begin{equation}\label{eq:f_flat} 
\xymatrix{
&& B^2 \ar[dr]^{m} \ar@{}[d]|{~\eqref{eq:component-compatibility}} \\
XB \ar@<2pt>[r]^-{w1} \ar@<-2pt>[r]_-{v1} & 
BAB \ar[ur]^{f_21} \ar[r]^{1f_1} \ar[dr]_{g_11} & B^2 \ar[r]^{m} & B \\
&& AB \ar[ur]_{f_1} }
\end{equation}
the lower region on the right commutes since $f\bullet g=i$. Since $w1$ and
$v1$ agree when composed with the lower path, they agree when composed with
the upper path. By non-degeneracy of the multiplication, it follows that
$f_2.w=f_2.v$, and so there is a unique $f^\flat\colon A\to B$
satisfying $f^\flat .g_1=f_2$.  

Using \eqref{eq:component-compatibility} together with the associativity and
the non-degeneracy of the multiplication, commutativity of the lower triangle
of \eqref{eq:f_flat} is seen to be equivalent to the commutativity of the
region marked by $(\ast)$ in
$$
\xymatrix{
(BA)^2 \ar[rrr]^-{11g_1} \ar[d]_-{g_1g_1} \ar[rd]^-{f_211} &&&
BA^2 \ar[rr]^-{g_11} \ar[rd]^-{1m} \ar[ld]_-{f_21} 
\ar@{}[dd]|-{\eqref{eq:multiplicative}} &
\ar@{}[d]^-{\eqref{eq:multiplier}}&
A^2 \ar[d]^-m \\
A^2 \ar[d]_-{f^\flat f^\flat} &
B^2A \ar[ld]^-{1f_2} \ar[r]^-{1g_1} \ar@{}[rd]|-{(\ast)} &
BA \ar[rd]_-{f_2} &&
BA \ar[ld]^-{f_2} \ar[r]^-{g_1} &
A \ar[d]^-{f^\flat} \\
B^2 \ar[rrr]_-m &&&
B \ar@{=}[rr] &&
B.}
$$
Since $g_1g_1\colon(BA)^2\to A^2$ is epi, commutativity of this proves that
$f^\flat$ is multiplicative. Using that $g_11\colon BAB \to AB$ is epi, it follows
by the commutativity of the square in \eqref{eq:f_flat} that $m.f^\flat 1=f_1$
so that $(f^\flat)^\#=f$.  

\endproof


\section{Monoidality}

In any braided monoidal category, the monoidal product of semigroups $A$ and $B$
is again a semigroup with multiplication
\begin{equation}\label{eq:AB-product}
\xymatrix{
(AB)^2 \ar[r]^-{1c1} &
A^2B^2 \ar[r]^-{mm} &
AB.}
\end{equation}
Our aim is to extend this construction to a monoidal structure on the category
$\cm$ of Proposition~\ref{prop:category-M}. While the category \cm is
also available for not necessarily braided monoidal categories \cc (see
Remark \ref{rem:no-braid}), its monoidal structure makes essential use of the
braiding of \cc.

\begin{proposition}\label{prop:non-d-tensor}
If the semigroups $A$ and $B$ are non-degenerate, then so is their monoidal  product $AB$. 
\end{proposition}

\proof
We check that the map sending a morphism $s\colon X\to AB$ to the morphism
$\Phi(s)\colon ABX\to AB$ given by  
$$\xymatrix{
ABX \ar[r]^{11s} & (AB)^2 \ar[r]^{1c1} & A^2B^2 \ar[r]^{mm} & AB }$$
is injective; the other half holds dually. By non-degeneracy of $A$, if we
know $\Phi(s)$ then we know the upper, and also the lower, composite in the
diagram 
$$\xymatrix{
BX \ar[r]^{1s} & BAB \ar[dr]_{1c^{-1}} \ar@{=}[rr] && BAB \ar[r]^{c1} & 
AB^2 \ar[r]^{1m} & AB \\
&& B^2A \ar[ur]_{1c} \ar[rr]_{m1} && BA \ar[ur]_c 
}$$
but now by non-degeneracy of $B$ we know $c^{-1}.s$ and so in turn we know $s$
as required.
\endproof 

For semigroup morphims $f:A\to B$ and $f':A'\to B'$, also the monoidal
product $ff':AA'\to BB'$ is compatible with the multiplication
\eqref{eq:AB-product}. The components of $(ff')^\#$ are
$$
\xymatrix{
AA'BB' \ar[r]^-{ff'11} &
(BB')^2 \ar[r]^-{1c1} &
B^2B^{\prime 2} \ar[r]^-{mm'} &
BB' &
B^2B^{\prime 2} \ar[l]_-{mm'} &
(BB')^2 \ar[l]_-{1c1} &
BB'AA' \ar[l]_-{11ff'} 
},
$$
motivating the following construction for more general \mM-morphisms.

\begin{proposition}\label{prop:monoidal-product}
If $f\colon A\nto B$ and $f'\colon A'\nto B'$ are $\mM$-morphisms, then the
pair  
$$
\xymatrix{
AA' BB' \ar[r]^-{1c1} &
ABA' B' \ar[r]^-{f_1 f'_1} &
BB'
&&
BB' AA' \ar[r]^-{1c1} &
BAB' A' \ar[r]^-{f_2 f'_2} &
BB'}
$$
defines an $\mM$-morphism $AA'\nto BB'$, which is multiplicative or dense if
$f$ and $f'$ are so.  
\end{proposition}

\begin{proof}
The stated morphisms render commutative the diagram of
\eqref{eq:component-compatibility} by commutativity of 
$$
\xymatrix@C=40pt{
BB'AA'BB'\ar[r]^-{111c1}\ar[d]_-{1c111} &
BB'ABA'B' \ar[r]^-{11f_1f'_1} \ar[d]^-{1c_{B',AB}11} &
(BB')^2\ar[d]^-{1c1} \\
BAB'A'BB'\ar[r]^-{11c_{B'A',B}1} \ar[d]_-{f_2f'_211} &
BABB'A'B' \ar[r]^-{1f_11f'_1} \ar[d]^-{f_21f'_21} 
\ar@{}[rd]|-{\eqref{eq:component-compatibility}}&
B^2B^{\prime 2}\ar[d]^-{mm'} \\
(BB')^2 \ar[r]_-{1c1} &
B^2B^{\prime 2}\ar[r]_-{mm'} &
BB' .}
$$
When it comes to multiplicativity, in view of Remark
\ref{rem:components_equivalent}, it is enough to check the commutativity of
one of the diagrams in \eqref{eq:multiplicative}. In the case of the first
one, for example, it follows by the commutativity of 
$$
\xymatrix@C=40pt{
(AA')^2 BB' \ar[r]^-{111c1} \ar[d]_-{1c111} &
AA'ABA'B' \ar[r]^-{11f_1f'_1}\ar[d]^-{1c_{A',AB}11} &
AA'BB'\ar[d]^-{1c1} \\
A^2A^{\prime 2}BB' \ar[r]^-{11c_{A'A',B}1} \ar[d]_-{mm'11} &
A^2BA^{\prime 2}B' \ar[r]^-{1f_11f'_1} \ar[d]^-{m1m'1} 
\ar@{}[rd]|-{\eqref{eq:multiplicative}}&
ABA'B'\ar[d]^-{f_1f'_1} \\
AA'BB' \ar[r]_-{1c1} &
ABA'B' \ar[r]_-{f_1f'_1} &
BB' .}
$$
Finally if $f$ and $f'$ are dense, then $f_1$ and $f'_1$ are in \cq, and so
their  monoidal  product $f_1f'_1$ is in \cq, as is its compsite
$(ff')_1$ with $1c1$. 
\end{proof}

\begin{proposition}\label{prop:monoidality_of_M}
The category \cm is monoidal with respect to the usual monoidal product of
semigroups, and with monoidal product of morphisms given as in
Proposition~\ref{prop:monoidal-product}.  
\end{proposition}

\proof
The associativity and unit isomorphisms are inherited from \cc as in
Remark~\ref{rem:A->B}. The naturality of these isomorphisms follows from their
naturality in \cc, using the description in Remark~\ref{rem:A->B} 
of composition in \cm with $g^\#$ for an isomorphism $g$. It remains only to
check that the monoidal product is functorial.  

Given morphisms $f\colon A\nto B$, $g\colon B\nto C$, $f'\colon A'\nto B'$,
and $g'\colon B'\nto C'$ in \cm, the right vertical in the diagram 
$$\xymatrix @C3pc {
AA'BB'CC' \ar[r]^{111c1} \ar[d]_{1c111} & 
AA'BCB'C' \ar[r]^{11g_1g'_1} \ar[d]^{1c111} & 
AA'CC' \ar[dd]^{1c1} \\
ABA'B'CC' \ar[r]^{111c1} \ar[dd]_{f_1f'_111} & ABA'CB'C' \ar[d]^{11c11} \\
& ABCA'B'C' \ar[r]^{1g_11g'_1} \ar[d]^{f_11f'_11} & 
ACA'C' \ar[d]^{(g\bullet f)_1(g'\bullet f')_1} \\
BB'CC' \ar[r]_{1c1} & BCB'C' \ar[r]_{g_1g'_1} & CC' }$$
is the first component of $(g\bullet f)(g'\bullet f')$, but commutativity of
the diagram means that it satisfies the defining property of the first
component of $gg'\bullet ff'$. Thus the monoidal product preserves
composition; preservation of identities is straightforward.   
\endproof


\section{Multiplier bimonoids as comonoids}

One of several equivalent ways of describing bimonoids in a braided monoidal
category is to say that they are comonoids in the monoidal category of
monoids. Our aim is to give an analogous description of (certain)
multiplier bimonoids in \cite{BohmLack:braided_mba} as comonoids in the
monoidal category \cm. This allows us to define morphisms of these multiplier
bimonoids as comonoid morphisms.  

\begin{theorem} \label{thm:mbm_vs_comonoid}
Let $\cc$ be a braided monoidal category satisfying the standing assumptions
of Section~\ref{sect:assumptions}, and let \cm be the induced monoidal
category as in Propositions \ref{prop:category-M} and
\ref{prop:monoidality_of_M}.  For an object $A$ of  $\cm$, and for morphisms
$t_1,t_2\colon A^2\to A^2$ and $e\colon A \to I$ in $\cc$, the following assertions are
equivalent. 

(i) There is a comonoid in $\cm$ with counit $(e:A\to I \leftarrow A:e)$ and
comultiplication 
$$
\xymatrix{
(d_1:A^3\ar[r]^-{1c^{-1}} &
A^3 \ar[r]^-{t_1 1} &
A^3 \ar[r]^-{1c} &
A^3 \ar[r]^-{m1} &
A^2 &
A^3 \ar[l]_-{1m} &
A^3 \ar[l]_-{c1} &
A^3 \ar[l]_-{1t_2} &
A^3 \ar[l]_-{c^{-1} 1}:d_2)}.
$$

(ii) There is a multiplier bimonoid $(A,t_1,t_2,e)$ in $\cc$ such that 
\begin{itemize}
\item the resulting multiplication $e1.t_1=1e.t_2$  is equal to
  the given one $m\colon A^2\to A$,  
\item the counit $e$, and the morphisms $d_1$ and $d_2$ in part (i) lie in \cq.
\end{itemize}
\end{theorem}

\begin{proof}
Let us spell out what is being asserted in (i). An \mM-morphism $A\nto I$ is
just a morphism $e:A\to I$ in \cc; it will be multiplicative as
an \mM-morphism if and only if it is multiplicative as a \cc-morphism, in the
sense that  
\begin{equation}\label{eq:e-multiplicative} 
e.m=ee 
\end{equation}
and it will be dense if and only if it lies in \cq.  
Using the associativity and the non-degeneracy of the multiplication, we see
that the pair $(d_1,d_2)$ renders commutative
\eqref{eq:component-compatibility} if and only if 
\begin{equation}\label{eq:t_1-2_short_compatibility}
m1.1t_1=1m.t_21 ,
\end{equation}
and it renders commutative the first diagram in \eqref{eq:multiplicative},
meaning
\begin{equation}\label{eq:d-short-fusion}
  d_1.1d_1 = d_1.m11, 
\end{equation}
if and only if the `short fusion equation' 
\begin{equation}\label{eq:short_fusion}
m1.c^{-1}1.1t_1.c1.1t_1=t_1.m1
\end{equation}
holds. 

From \eqref{eq:t_1-2_short_compatibility} it follows that 
$$
\xymatrix{
A^4\ar[rr]^-{11m} \ar[dd]_-{1t_11} \ar[rd]^-{t_211}&&
A^3 \ar[r]^-{1t_1} \ar[d]^-{t_21} &
A^3 \ar[ddd]^-{m1}\\
\ar@{}[rd]|-{\eqref{eq:t_1-2_short_compatibility}} & 
A^4 \ar[r]^-{11m} \ar[d]^-{1m1} &
A^3 \ar[rdd]^-{1m} \ar@{}[r]|-{\eqref{eq:t_1-2_short_compatibility}} &\\
A^4 \ar[d]_-{11m} \ar[r]^-{m11} &
A^3 \ar[rrd]_-{1m} \ar@{}[rr]|(.42){\textrm{(associativity)}}&&\\
A^3 \ar[rrr]_-{m1} &&&
A^2}
$$
commutes; hence by the non-degeneracy of $m$,
\begin{equation}\label{eq:t_1_A-linear}
t_1.1m=1m.t_11
\end{equation}
or equivalently 
\begin{equation}\label{eq:d_1_A-linear}
  d_1.11m = 1m.d_1.
\end{equation}

The \mM-morphism $e:A \nto I$ is a left counit for the comultiplication $d:A
\nto A^2$ if and only if either (and hence by Remark
\ref{rem:components_equivalent} both) of the diagrams 
$$
\xymatrix{
A^4 \ar[r]^-{1em} \ar[d]^-{d_11} &
A^2 \ar[d]_-m
&
A^4 \ar[r]^-{c11} \ar[d]^-{1d_2} &
A^4 \ar[r]^-{em1} &
A^2 \ar[d]_-m \\
A^3 \ar[r]_-{em} &
A
&
A^3 \ar[r]_-{c1} &
A^3 \ar[r]_-{em} &
A}
$$
commutes. Using \eqref{eq:e-multiplicative}, the associativity and the
non-degeneracy of $m$, and the fact that $1e1$ is an epimorphism, they
are seen to be equivalent to
\begin{equation}\label{eq:t_1-m}
e1.t_1=m
\end{equation}
and
\begin{equation}\label{eq:t_2-counital}
e1.t_2=e1,
\end{equation}
respectively. Symmetrically, $e\colon A \nto I$ is a right counit if and only
if  
\begin{equation}\label{eq:t_2-m}
1e.t_2=m,
\end{equation}
equivalently,
\begin{equation}\label{eq:t_1-counital}
1e.t_1=1e.
\end{equation}
Finally, $d\colon A \nto A^2$ is coassociative if and only if the uniquely
determined right verticals of the the diagrams 
\begin{equation}\label{coasociativity-1}
\xymatrix{
A^6 \ar[r]^-{11c11} \ar[d]^-{d_1111} &
A^6 \ar[r]^-{111c1} &
A^6 \ar[r]^-{1d_1m} &
A^4 \ar@{-->}[d]
&
A^6 \ar[r]^-{11c11} \ar[d]^-{d_1111} &
A^6 \ar[r]^-{1md_1} &
A^4 \ar@{-->}[d] \\
A^5 \ar[r]_-{1c11} &
A^5 \ar[r]_-{11c1} &
A^5 \ar[r]_-{d_1m} &
A^3
&
A^5 \ar[r]_-{1c11} &
A^5 \ar[r]_-{md_1} &
A^3}
\end{equation}
are equal to each other. (Clearly there is an equivalent equation
involving $d_2$.) 

In the following string calculation, the first equality holds by
\eqref{eq:t_1_A-linear} and the second by \eqref{eq:d-short-fusion} 
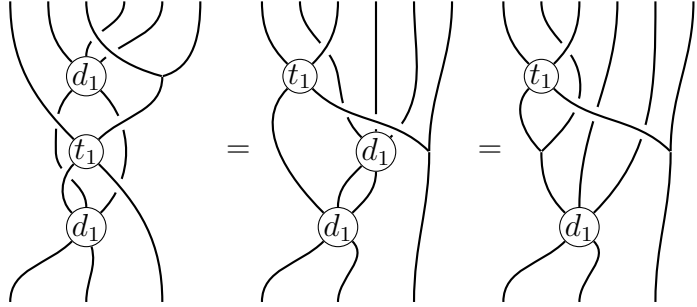
\begin{equation*}
  \begin{tikzpicture}
\path (1,2) node[arr,name=d1u] {$d_1$}
(1,1) node[arr,name=t1] {$t_1$}
(0.6,1) node[empty,name=x] {} 
(1,0) node[arr,name=d1l] {$d_1$}
(2,2) node[empty,name=m] {}
(0,3) node[empty,name=in1] {}
(0.5,3) node[empty,name=in2] {}
(1,3) node[empty,name=in3] {} 
(1.5,3) node[empty,name=in4] {}
(2,3) node[empty,name=in5] {}
(2.5,3) node[empty,name=in6] {}
(0,-1) node[empty,name=b1] {}
(1,-1) node[empty,name=b2] {}
(2,-1) node[empty,name=b3] {};
\path[braid, name path=bin1] (in1) to[out=270,in=135] (t1);
\draw[braid] (in2) to[out=270,in=135] (d1u);
\path[braid,name path=bin3] (in3) to[out=270,in=165] (m);
\draw[braid,name path=bin4] (in4) to[out=270,in=90] (d1u);
\draw[braid,name path=bin5] (in5) to[out=270,in=45] (d1u);
\fill[white,name intersections={of=bin3 and bin4}] (intersection-1) circle(0.1);
\fill[white,name intersections={of=bin3 and bin5}] (intersection-1) circle(0.1);
\draw[braid] (in3) to[out=270,in=165] (m);
\draw[braid] (in6) to[out=270,in=15] (m);
\path[braid, name path=mt] (m) to[out=270,in=45] (t1);
\path[braid, name path=bb3] (t1) to[out=315,in=90] (b3);
\path[braid, name path=td] (t1) to[out=225,in=135] (d1l);
\draw[braid, name path=dd1] (d1u) to[out=225,in=90] (x) to[out=270,in=90] (d1l);
\fill[white,name intersections={of=bin1 and dd1}] (intersection-1) circle(0.1);
\draw[braid] (in1) to[out=270,in=135] (t1);
\fill[white,name intersections={of=td and dd1}] (intersection-1) circle(0.1);
\draw[braid] (t1) to[out=225,in=135] (d1l);
\draw[braid, name path=dd2] (d1u) to[out=315,in=45] (d1l);
\fill[white,name intersections={of=mt and dd2}] (intersection-1) circle(0.1);
\draw[braid] (m) to[out=270,in=45] (t1);
\fill[white,name intersections={of=bb3 and dd2}] (intersection-1) circle(0.1);
\draw[braid] (t1) to[out=315,in=90] (b3);
\draw[braid] (d1l) to[out=285,in=90] (b2);
\draw[braid] (d1l) to[out=225,in=90] (b1);
\draw (3,1) node[empty] {$=$};
  \end{tikzpicture}
  \begin{tikzpicture}
\path (1.5,1) node[arr,name=d1u] {$d_1$}
(0.5,2) node[arr,name=t1] {$t_1$}
(1,2) node[empty,name=x] {} 
(1,0) node[arr,name=d1l] {$d_1$}
(2.2,1) node[empty,name=m] {}
(0,3) node[empty,name=in1] {}
(0.5,3) node[empty,name=in2] {}
(1,3) node[empty,name=in3] {} 
(1.5,3) node[empty,name=in4] {}
(2,3) node[empty,name=in5] {}
(2.5,3) node[empty,name=in6] {}
(0,-1) node[empty,name=b1] {}
(1,-1) node[empty,name=b2] {}
(2,-1) node[empty,name=b3] {};
\draw[braid] (in1) to[out=270,in=135] (t1);
\draw[braid,name path=bin2] (in2) to[out=270,in=100] (x) to[out=280,in=135] (d1u);
\path[braid,name path=bin3] (in3) to[out=270,in=45] (t1);
\fill[white,name intersections={of=bin2 and bin3}] (intersection-1) circle(0.1);
\draw[braid] (in3) to[out=270,in=45] (t1);
\draw[braid,name path=bin4] (in4) to[out=270,in=90] (d1u);
\draw[braid,name path=bin5] (in5) to[out=270,in=45] (d1u);
\draw[braid,name path=bin6] (in6) to[out=270,in=90] (m);
\draw[braid] (m) to[out=270,in=90] (b3);
\draw[braid] (t1) to[out=225,in=135] (d1l);
\path[braid,name path=tm] (t1) to[out=315,in=135] (m);
\fill[white,name intersections={of=tm and bin2}] (intersection-1) circle(0.1);
\fill[white,name intersections={of=tm and bin4}] (intersection-1) circle(0.1);
\fill[white,name intersections={of=tm and bin5}] (intersection-1) circle(0.1);
\draw[braid] (t1) to[out=315,in=135] (m);
\draw[braid] (d1u) to[out=225,in=90] (d1l);
\draw[braid] (d1u) to[out=270,in=45] (d1l);
\draw[braid] (d1l) to[out=225,in=90] (b1);
\draw[braid] (d1l) to[out=315,in=90] (b2);
\draw (3,1) node[empty] {$=$};
  \end{tikzpicture}
  \begin{tikzpicture}
\path (0.5,1) node[empty,name=m1] {}
(0.5,2) node[arr,name=t1] {$t_1$}
(1,2) node[empty,name=x] {} 
(1,0) node[arr,name=d1l] {$d_1$}
(2.2,1) node[empty,name=m] {}
(0,3) node[empty,name=in1] {}
(0.5,3) node[empty,name=in2] {}
(1,3) node[empty,name=in3] {} 
(1.5,3) node[empty,name=in4] {}
(2,3) node[empty,name=in5] {}
(2.5,3) node[empty,name=in6] {}
(0,-1) node[empty,name=b1] {}
(1,-1) node[empty,name=b2] {}
(2,-1) node[empty,name=b3] {};
\draw[braid] (in1) to[out=270,in=135] (t1);
\draw[braid,name path=bin2] (in2) to[out=270,in=100] (x) to[out=280,in=45] (m1);
\path[braid,name path=bin3] (in3) to[out=270,in=45] (t1);
\fill[white,name intersections={of=bin2 and bin3}] (intersection-1) circle(0.1);
\draw[braid] (in3) to[out=270,in=45] (t1);
\draw[braid,name path=bin4] (in4) to[out=270,in=90] (d1l);
\draw[braid,name path=bin5] (in5) to[out=270,in=45] (d1l);
\draw[braid,name path=bin6] (in6) to[out=270,in=90] (m);
\draw[braid] (m) to[out=270,in=90] (b3);
\draw[braid] (t1) to[out=225,in=135] (m1);
\draw[braid] (m1) to[out=270,in=135] (d1l);
\path[braid,name path=tm] (t1) to[out=315,in=135] (m);
\fill[white,name intersections={of=tm and bin2}] (intersection-1) circle(0.1);
\fill[white,name intersections={of=tm and bin4}] (intersection-1) circle(0.1);
\fill[white,name intersections={of=tm and bin5}] (intersection-1) circle(0.1);
\draw[braid] (t1) to[out=315,in=135] (m);
\draw[braid] (d1l) to[out=225,in=90] (b1);
\draw[braid] (d1l) to[out=315,in=90] (b2);
  \end{tikzpicture}
\end{equation*}
and the result is that the unique morphism rendering commutative
the first diagram in \eqref{coasociativity-1} is 
\begin{equation}\label{eq:right-vertical}
\xymatrix{
A^4 \ar[r]^-{cc^{-1}} &
A^4 \ar[r]^-{1t_11} &
A^4 \ar[r]^-{c^{-1}c} &
A^4 \ar[r]^-{d_11} &
A^3 .}
\end{equation}
Therefore $d\colon A \nto A^2$ is coassociative if and only if the second
diagram in \eqref{coasociativity-1} commutes, with the morphism
\eqref{eq:right-vertical} in the right vertical. 

By naturality, coherence, and by the associativity of $m$, it follows that 
$$
\xymatrix{
A^4\ar[r]^-{11c} \ar[d]_-{d_11} &
A^4 \ar[r]^-{1m1} &
A^3 \ar[d]^-{d_1} \\
A^3 \ar[r]_-{1c} &
A^3 \ar[r]_-{m1} &
A^2}
$$
commutes. Applying this together with \eqref{eq:t_1_A-linear} and using the
non-degeneracy of $m$, we see that commutativity of the second diagram in
\eqref{coasociativity-1}, with \eqref{eq:right-vertical} in the right
vertical, is equivalent to the `fusion equation'
\begin{equation}\label{eq:fusion}
t_11.1c.t_11.1c^{-1}.1t_1=1t_1.t_11.
\end{equation}

Summarizing, we proved so far that assertion (i) is equivalent to the validity
of \eqref{eq:e-multiplicative}, \eqref{eq:t_1-2_short_compatibility},
\eqref{eq:short_fusion}, \eqref{eq:t_1-m}, \eqref{eq:t_2-m}, and
\eqref{eq:fusion}. 

On the other hand, it was shown in \cite[Proposition
3.7]{BohmLack:braided_mba} that assertion (ii) is equivalent to the validity
of \eqref{eq:e-multiplicative}, \eqref{eq:t_1-m}, \eqref{eq:t_2-m},
\eqref{eq:fusion}, and the compatibility condition 
\begin{equation}\label{eq:t_1-2_compatibility}
t_21.1t_1=1t_1.t_21.
\end{equation}
Thus it remains to show that, in the presense of 
\eqref{eq:e-multiplicative}, \eqref{eq:t_1-m}, \eqref{eq:t_2-m}, and
\eqref{eq:fusion}, the condition \eqref{eq:t_1-2_compatibility} is equivalent
to the conjunction of \eqref{eq:t_1-2_short_compatibility} and
\eqref{eq:short_fusion}. 

For the forward implication, condition \eqref{eq:t_1-2_short_compatibility}
holds by (3.2) in 
\cite{BohmLack:braided_mba}, and \eqref{eq:short_fusion} holds by
\cite[Remark~3.6]{BohmLack:braided_mba}. For the converse,
first observe that, by \eqref{eq:t_1_A-linear} and non-degeneracy, the fusion
equation \eqref{eq:fusion} is equivalent to  
\begin{equation}\label{eq:d-fusion}
\begin{tikzpicture}
  \path (1,2) node[arr,name=t1u] {$d_1$} 
(1,1) node[arr,name=t1d]  {$t_1$} 
(0,0.3) node[arr,name=t1l] {$t_1$}; 
\path (1.6,2.7) node[empty,name=s4t] {}
(1.1,2.7) node[empty,name=s3t] {} 
(0.5,2.7) node[empty,name=s2t] {} 
(0,2.7) node[empty,name=s1t] {} 
(-0.5,-0.5) node[empty,name=s1b] {} 
(0.5,-0.5) node[empty,name=s2b] {} 
(1,-0.5) node[empty,name=s3b] {}; 
\draw[braid] (s4t) to[out=270,in=45] (t1u);
\draw[braid] (s3t) to[out=270,in=90] (t1u);
\draw[braid] (s2t) to[out=270,in=135] (t1u);
\draw[braid] (t1u) to[out=315,in=45] (t1d);
\draw[braid] (t1d) to[out=315,in=45] (s3b);
\path[braid,name path=s4] (s1t) to[out=270,in=135] (t1d);
\draw[braid,name path=s5] (t1u) to[out=225,in=0] (t1l);
\fill[white, name intersections={of=s4 and s5}] (intersection-1) circle(0.1);
\path[braid,name path=s6] (t1d) to[out=210,in=0] (0,0.8) [out=180,in=180] (t1l);
\fill[white, name intersections={of=s6 and s5}] (intersection-1) circle(0.1);
\draw[braid,name path=s4] (s1t) to[out=270,in=135] (t1d);
\draw[braid,name path=s6] (t1d) to[out=210,in=0] (0,0.8) to[out=180,in=180] (t1l);
\draw[braid] (t1l) to[out=225,in=90] (s1b);
\draw[braid] (t1l) to[out=315,in=90] (s2b);
\draw (2,1.35) node {$=$}; 
\end{tikzpicture}
\begin{tikzpicture}
  \path (0.75,1.0) node[arr,name=t1d] {$d_1$} 
(0.25,2) node[arr,name=t1u] {$t_1$} ;
\path (1.5,2.7) node[empty,name=s4t] {};
(1.0,2.7) node[empty,name=s3t] {} 
(0.5,2.7) node[empty,name=s2t] {} 
(-0.25,-0.5) node[empty,name=s1b] {} 
(0.5,-0.5) node[empty,name=s2b] {} 
(1,-0.5) node[empty,name=s3b] {} 
(-0.5,2.7) node[empty,name=s0t] {};
\draw[braid] (s4t) to[out=270,in=45] (t1d);
\draw[braid] (s3t) to[out=270,in=90]  (t1d);
\draw[braid] (s2t) to[out=270,in=45] (t1u); 
\draw[braid] (s1t) to[out=315,in=135] (t1u);
\draw[braid] (t1u) to[out=225,in=90] (s1b);
\draw[braid] (t1u) to[out=315,in=135] (t1d);
\draw[braid] (t1d) to[out=250,in=100]  (s2b);
\draw[braid] (t1d) to[out=290,in=80] (s3b);
\end{tikzpicture}
\end{equation}
and now
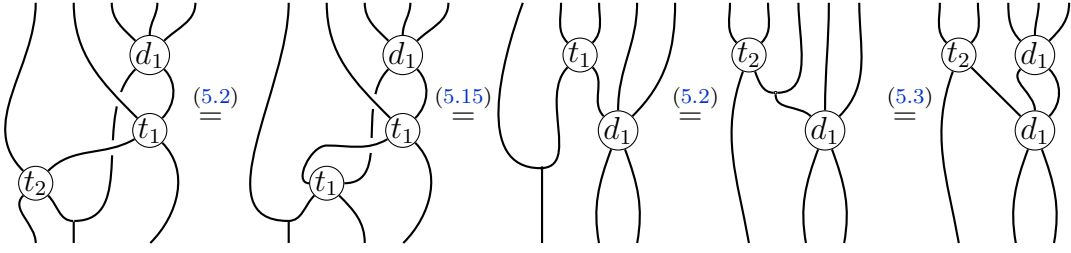
\begin{equation*}
 \begin{tikzpicture}
  \path (1,2) node[arr,name=t1u] {$d_1$} 
(1,1) node[arr,name=t1d]  {$t_1$} 
(-0.5,0.3) node[arr,name=t1l] {$t_2$} 
(-0,-0.2) node[empty,name=m] {};
\path (1.6,2.7) node[empty,name=s4t] {}
(1.1,2.7) node[empty,name=s3t] {} 
(0.5,2.7) node[empty,name=s2t] {} 
(0,2.7) node[empty,name=s1t] {} 
(-0.5,-0.5) node[empty,name=s1b] {} 
(0.,-0.5) node[empty,name=s2b] {} 
(1,-0.5) node[empty,name=s3b] {} 
(-0.5,2.7) node[empty,name=s0t] {};
\draw[braid] (s4t) to[out=270,in=45] (t1u);
\draw[braid] (s3t) to[out=270,in=90] (t1u);
\draw[braid] (s2t) to[out=270,in=135] (t1u);
\draw[braid] (t1u) to[out=315,in=45] (t1d);
\draw[braid] (t1d) to[out=315,in=45] (s3b);
\path[braid,name path=s4] (s1t) to[out=270,in=135] (t1d);
\draw[braid,name path=s5] (t1u) to[out=225,in=0] (m);
\fill[white, name intersections={of=s4 and s5}] (intersection-1) circle(0.1);
\path[braid,name path=s6] (t1d) to[out=210,in=45]  (t1l);
\fill[white, name intersections={of=s6 and s5}] (intersection-1) circle(0.1);
\draw[braid,name path=s4] (s1t) to[out=270,in=135] (t1d);
\draw[braid] (t1d) to[out=210,in=45]  (t1l);
\draw[braid] (m) to (s2b);
\draw[braid] (s0t) to[out=270,in=135] (t1l);
\draw[braid] (t1l) to[out=315,in=180] (m);
\draw[braid] (t1l) to[out=225,in=90] (s1b);
\draw (1.8,1.35) node[empty] {$\stackrel{~\eqref{eq:t_1-2_short_compatibility}}=$};
\end{tikzpicture} \hskip-0.5cm
 \begin{tikzpicture}
  \path (1,2) node[arr,name=t1u] {$d_1$} 
(1,1) node[arr,name=t1d]  {$t_1$} 
(0,0.3) node[arr,name=t1l] {$t_1$} 
(-0.5,-0.2) node[empty,name=m] {};
\path (1.6,2.7) node[empty,name=s4t] {}
(1.1,2.7) node[empty,name=s3t] {} 
(0.5,2.7) node[empty,name=s2t] {} 
(0,2.7) node[empty,name=s1t] {} 
(-0.5,-0.5) node[empty,name=s1b] {} 
(0.5,-0.5) node[empty,name=s2b] {} 
(1,-0.5) node[empty,name=s3b] {} 
(-0.5,2.7) node[empty,name=s0t] {};
\draw[braid] (s4t) to[out=270,in=45] (t1u);
\draw[braid] (s3t) to[out=270,in=90] (t1u);
\draw[braid] (s2t) to[out=270,in=135] (t1u);
\draw[braid] (t1u) to[out=315,in=45] (t1d);
\draw[braid] (t1d) to[out=315,in=45] (s3b);
\path[braid,name path=s4] (s1t) to[out=270,in=135] (t1d);
\draw[braid,name path=s5] (t1u) to[out=225,in=0] (t1l);
\fill[white, name intersections={of=s4 and s5}] (intersection-1) circle(0.1);
\path[braid,name path=s6] (t1d) to[out=210,in=0] (0,0.8) [out=180,in=180] (t1l);
\fill[white, name intersections={of=s6 and s5}] (intersection-1) circle(0.1);
\draw[braid,name path=s4] (s1t) to[out=270,in=135] (t1d);
\draw[braid,name path=s6] (t1d) to[out=210,in=0] (0,0.8) to[out=180,in=180] (t1l);
\draw[braid] (m) to (s1b);
\draw[braid] (s0t) to[out=270,in=180] (m) to[out=0,in=225] (t1l);
\draw[braid] (t1l) to[out=315,in=90] (s2b);
\draw (1.8,1.35) node[empty] {$\stackrel{\eqref{eq:d-fusion}}=$};
\end{tikzpicture} \hskip-0.25cm
\begin{tikzpicture}
  \path (0.75,1.0) node[arr,name=t1d] {$d_1$} 
(0.25,2) node[arr,name=t1u] {$t_1$} 
(-0.25,0.5) node[arr,name=m]  {} ;
\path (1.5,2.7) node[empty,name=s4t] {}
(1.0,2.7) node[empty,name=s3t] {} 
(0.5,2.7) node[empty,name=s2t] {} 
(0,2.7) node[empty,name=s1t] {} 
(-0.25,-0.5) node[empty,name=s1b] {} 
(0.5,-0.5) node[empty,name=s2b] {} 
(1,-0.5) node[empty,name=s3b] {} 
(-0.5,2.7) node[empty,name=s0t] {};
\draw[braid] (s4t) to[out=270,in=45] (t1d);
\draw[braid] (s3t) to[out=270,in=90]  (t1d);
\draw[braid] (s2t) to[out=270,in=45] (t1u); 
\draw[braid] (s1t) to[out=270,in=135] (t1u);
\draw[braid] (s0t) to[out=250,in=180] (m) to[out=0,in=225] (t1u);
\draw[braid] (m) to[out=270,in=90] (s1b);
\draw[braid] (t1u) to[out=315,in=135] (t1d);
\draw[braid] (t1d) to[out=250,in=100]  (s2b);
\draw[braid] (t1d) to[out=290,in=80] (s3b);
\draw (1.7,1.35) node[empty] {$\stackrel{~\eqref{eq:t_1-2_short_compatibility}}=$};
\end{tikzpicture} \hskip-0.00cm
\begin{tikzpicture}
  \path (0.75,1.0) node[arr,name=d1] {$d_1$} 
(-0.25,2) node[arr,name=t2] {$t_2$} 
(0.1,1.5) node[arr,name=m]  {} ;
\path (1.2,2.7) node[empty,name=s4t] {}
(0.8,2.7) node[empty,name=s3t] {} 
(0.4,2.7) node[empty,name=s2t] {} 
(0,2.7) node[empty,name=s1t] {} 
(-0.25,-0.5) node[empty,name=s1b] {} 
(0.5,-0.5) node[empty,name=s2b] {} 
(1,-0.5) node[empty,name=s3b] {} 
(-0.5,2.7) node[empty,name=s0t] {};
\draw[braid] (s4t) to[out=270,in=45] (d1);
\draw[braid] (s3t) to[out=270,in=90]  (d1);
\draw[braid] (s2t) to[out=270,in=0] (m); 
\draw[braid] (s1t) to[out=270,in=45] (t2);
\draw[braid] (s0t) to[out=270,in=135] (t2);
\draw[braid] (t2) to[out=250,in=100] (s1b);
\draw[braid] (m) to[out=270,in=135] (d1);
\draw[braid] (t2) to[out=290,in=180] (m);
\draw[braid] (d1) to[out=250,in=100]  (s2b);
\draw[braid] (d1) to[out=290,in=80] (s3b);
\draw (1.8,1.35) node {$\stackrel{~\eqref{eq:d-short-fusion}}=$};
\end{tikzpicture} \hskip -0.2cm
\begin{tikzpicture}
  \path (0.75,2.0) node[arr,name=d1u] {$d_1$} 
(-0.25,2) node[arr,name=t2] {$t_2$} 
(0.75,1.0) node[arr,name=d1d]  {$d_1$} ;
\path (1.2,2.7) node[empty,name=s4t] {}
(0.8,2.7) node[empty,name=s3t] {} 
(0.4,2.7) node[empty,name=s2t] {} 
(0,2.7) node[empty,name=s1t] {} 
(-0.25,-0.5) node[empty,name=s1b] {} 
(0.5,-0.5) node[empty,name=s2b] {} 
(1,-0.5) node[empty,name=s3b] {} 
(-0.5,2.7) node[empty,name=s0t] {};
\draw[braid] (s4t) to[out=270,in=45] (d1u);
\draw[braid] (s3t) to[out=270,in=90]  (d1u);
\draw[braid] (s2t) to[out=270,in=135] (d1u); 
\draw[braid] (s1t) to[out=270,in=45] (t2);
\draw[braid] (s0t) to[out=270,in=135] (t2);
\draw[braid] (t2) to[out=250,in=100] (s1b);
\draw[braid] (t2) to[out=315,in=135] (d1d);
\draw[braid] (d1u) to[out=315,in=45] (d1d);
\draw[braid] (d1u) to[out=225,in=90] (d1d);
\draw[braid] (d1d) to[out=250,in=100]  (s2b);
\draw[braid] (d1d) to[out=290,in=80] (s3b);
\end{tikzpicture}
\end{equation*}
and cancelling $d_1$ from the left and right hand composite and using
non-degeneracy now gives the desired result.  
\end{proof}

The result \cite[Proposition~3.1]{JanssenVercruysse:mba&mha} can be seen as
the special case where \cc is the category of modules over a commutative ring,
and where the object $A$ is projective over that ring: it then states that if
$A$ is a multiplier Hopf algebra in the sense of Van Daele
\cite{VanDaele:multiplier_Hopf} then it can be seen as a comooid in the
corresponding category \cm. Theorem~\ref{thm:mbm_vs_comonoid} shows that the
restriction to projective modules can be avoided, as well as generalizing to
other braided monoidal categories. 

Let us stress that in Theorem \ref{thm:mbm_vs_comonoid} we only
described {\em certain} multiplier bimonoids as comonoids in $\cm$ (those
whose multiplication is non-degenerate, and for which the multiplication as
well as morphisms $d_1,d_2$ and $e$ belong to \cq). Also, {\em not every}
comonoid in $\cm$ corresponds to a multiplier bimonoid (only those whose
morphisms $d_1,d_2$ have a particular form). Results stronger in both aspects
can be achieved by taking a different point of view. Recall that comonoids in
a monoidal category $\mathcal M$ can be regarded as simplicial maps from the
Catalan simplicial set $\mathbb C$ to the nerve of $\mathcal M^{\mathsf{co}}$
(meaning the category with the reverse composition)
\cite{BuckGarnLackStreet}. In \cite{BohmLack:simplicial} we construct a
simplicial set which is not necessarily the nerve of any monoidal category,
but for which the simplicial maps from $\mathbb C$ to it can be identified
with multiplier bimonoids.  


\section{Morphisms}

We have seen how to identify (certain) multiplier bimonoids in the braided
monoidal category \cc with comonoids in the monoidal category \cm. We shall
now investigate morphisms of comonoids. 

\subsection{Morphisms between comonoids in \cm}
Suppose that $(C,d,e)$ and $(C',d',e')$ are comonoids in \cm. We claim
that a  morphism of comonoids is then a morphism $f:C \nto C'$ in $\cm$ 
whose components render commutative the following diagrams. 
\begin{equation}\label{eq:comonoidal-morphism}
\xymatrix{
CC'\ar[r]^-{f_1} \ar[rd]_-{ee'} &
C'\ar[d]^-{e'}
&
C'C^3C^{\prime 2}  \ar[r]^-{1d_111} \ar[d]_-{111c1} &
C'C^2C^{\prime 2} \ar[r]^-{11c1} &
C'(CC')^2 \ar[r]^-{1f_1f_1} &
C^{\prime 3} \ar[d]^-{d'_1}\\
& I
&
C'C(CC')^2\ar[rr]_-{f_2f_1f_1} &&
C^{\prime 3} \ar[r]_-{d'_1} & 
C^{\prime 2}}
\end{equation}
There is also of course an equivalent, symmetric set of diagrams with the
roles of the components interchanged:
$$
\xymatrix{
C'C\ar[r]^-{f_2} \ar[rd]_-{e'e} &
C'\ar[d]^-{e'}
&
C^{\prime 2}C^3C' \ar[r]^-{11d_21} \ar[d]_-{1c111} &
C^{\prime 2}C^2C' \ar[r]^-{1c11} &
(C'C)^2C' \ar[r]^-{f_2f_21} &
C^{\prime 3} \ar[d]^-{d'_2}\\
& I
&
(C'C)^2CC'\ar[rr]_-{f_2f_2f_1} &&
C^{\prime 3} \ar[r]_-{d'_2} & 
C^{\prime 2} .}
$$
Moreover, using the non-degeneracy of $C^{\prime 2}$, the second diagram of
\eqref{eq:comonoidal-morphism} is seen to be equivalent also to either of the
symmetric diagrams 
$$
\xymatrix@C=15pt{
C^{\prime 3}C^3\ar[rr]^-{d'_2d_1} \ar[d]_-{11f_211} &&
C^{\prime 2}C^2\ar[r]^-{1c1} &
(C'C)^2 \ar[d]^-{f_2f_2} 
&
C^3C^{\prime 3}\ar[rr]^-{d_2d'_1} \ar[d]_-{11f_111} &&
C^2C^{\prime 2}\ar[r]^-{1c1} &
(CC')^2 \ar[d]^-{f_1f_1} \\
C^{\prime 3}C^2\ar[r]_-{d'_211} &
C^{\prime 2}C^2 \ar[r]_-{1c1} &
(C'C)^2\ar[r]_-{f_2f_2} &
C^{\prime 2}
&
C^2C^{\prime 3}\ar[r]_-{11d'_1} &
C^2C^{\prime 2}\ar[r]_-{1c1} &
(CC')^2\ar[r]_-{f_1f_1} &
C^{\prime 2} .}
$$

We now explain why \eqref{eq:comonoidal-morphism} is equivalent to
preservation of the comonoid structure. Counitality of $f$ is clearly
equivalent to commutativity of the first diagram of
\eqref{eq:comonoidal-morphism}; and $f$ is comultiplicative if and only if the
uniquely determined right verticals of the 
diagrams \begin{equation} \label{eq:f_comultiplicative}
\xymatrix{
CC^{\prime 3} \ar[r]^-{1d'_1} \ar[d]_-{f_111} &
CC^{\prime 2} \ar@{-->}[d]
&&
C^3 C^{\prime 2} \ar[r]^-{11c1} \ar[d]_-{d_111} &
C(CC')^2 \ar[r]^-{1f_1f_1} &
CC^{\prime 2} \ar@{-->}[d]\\
C^{\prime 3} \ar[r]_-{d'_1} &
C^{\prime 2}
&&
C^2C^{\prime 2} \ar[r]_-{1c1} &
(CC')^2 \ar[r]_-{f_1f_1} &
C^{\prime 2}}
\end{equation}
are equal to each other. Let us denote this common morphism by $g\colon CC^{\prime
2} \to C^{\prime 2}$. Using the non-degeneracy of the multiplication 
of $C^{\prime 2}$ and the fact that $d'_21111\colon C^{\prime 3} C C^{\prime 3} \to
C^{\prime 2} C C^{\prime 3}$ is an epimorphism, commutativity of the first
diagram is equivalent to commutativity of  
$$
\xymatrix{
C^{\prime 3} C C^{\prime 3} \ar[r]^-{d'_21111} \ar[d]_-{d'_21111} &
C^{\prime 2} C C^{\prime 3} \ar[r]^-{111 d'_1} &
C^{\prime 2} C C^{\prime 2} \ar[r]^-{11g} &
C^{\prime 4} \ar[r]^-{1c1} &
C^{\prime 4} \ar[d]^-{m'm'} \\
C^{\prime 2}CC^{\prime 3}  \ar[r]_-{11f_111} &
C^{\prime 5}  \ar[r]_-{11d'_1} &
C^{\prime 4} \ar[r]_-{1c1} &
C^{\prime 4} \ar[r]_-{m'm'} &
C^{\prime 2} }
$$
and by \eqref{eq:component-compatibility} and the non-degeneracy of $C^{\prime
2}$ again, this is further equivalent to commutaivity of  
$$
\xymatrix{
C'CC^{\prime 3} \ar[r]^-{11d'_1} \ar[d]_-{1f_111} &
C'CC^{\prime 2} \ar[r]^-{1g} &
C^{\prime 3} \ar[d]^-{d'_1} \\
C^{\prime 4}\ar[r]_-{1d'_1} &
C^{\prime 3}\ar[r]_-{d'_1} &
C^{\prime 2}.}
$$
By commutativity of the diagram 
$$
\xymatrix{
C'CC^{\prime 3} \ar[rr]^-{1f_111} \ar[dr]^-{f_2111} \ar[dd]_{11d'_1} &
\ar@{}[rd]|-{\eqref{eq:component-compatibility}} &
C^{\prime 4} \ar[r]^-{1d'_1} \ar[d]^-{m'11} &
C^{\prime 3} \ar[dd]^-{d'_1} \\
& C^{\prime 4} \ar[r]^-{m'11} \ar[d]^-{1d'_1} &
C^{\prime 3} \ar[rd]^(.6){d'_1} 
\ar@{}[r]|-{\eqref{eq:multiplicative}}
\ar@{}[d]|-{\eqref{eq:multiplicative}} & \\
C'CC^{\prime 2} \ar[r]_{f_211} & 
C^{\prime 3} \ar[rr]_-{d'_1} &&
C^{\prime 2} .}
$$
and the fact that $11d'_1$ is an epimorphism, commutativity of the first
diagram in 
\eqref{eq:f_comultiplicative} is further equivalent to commutativity of 
\begin{equation}\label{eq:f_comultiplicative-1}
\xymatrix{
C'CC^{\prime 2} \ar[r]^-{1g} \ar[d]_-{f_211}&
C^{\prime 3} \ar[d]^-{d'_1} \\
C^{\prime 3}\ar[r]_-{d'_1} &
C^{\prime 2} .}
\end{equation}
We conclude that $f$ is comultiplicative if and only if the second diagram in
\eqref{eq:f_comultiplicative} commutes, having in the right vertical the
unique morphism $g$ rendering commutative \eqref{eq:f_comultiplicative-1}.

By similar steps to those used to analyze the first diagram, commutativity of
the second diagram in \eqref{eq:f_comultiplicative} is seen to be equivalent
to commutativity of    
$$
\xymatrix{
C'C^3C^{\prime 2} \ar[r]^-{111c1}\ar[d]_-{1d_111} &
C'C(CC')^2 \ar[r]^-{11f_1f_1} &
C'CC^{\prime 2} \ar[r]^-{1g} &
C^{\prime 3}\ar[d]^-{d'_1}\\
C'C^2C^{\prime 2}\ar[r]_-{11c1} &
C'(CC')^2 \ar[r]_-{1f_1f_1} &
C^{\prime 3}\ar[r]_-{d'_1} &
C^{\prime 2} }
$$
or, writing the top right path in an equal form via
\eqref{eq:f_comultiplicative-1}, 
to commutativity of the second diagram in \eqref{eq:comonoidal-morphism}.

\subsection{Morphisms between multiplier bimonoids.}
We may define morphisms between the multiplier bimonoids in part (ii) of 
Theorem \ref{thm:mbm_vs_comonoid} as comonoid morphisms between the
corresponding comonoids in part (i) of  Theorem \ref{thm:mbm_vs_comonoid}.  
This leads to the following explicit description:

Let  $(A,t_1,t_2,e)$ and $(A',t'_1,t'_2,e')$ be
bimonoids in $\cc$ obeying the conditions in Theorem
\ref{thm:mbm_vs_comonoid}~(ii). 
We claim that a morphism of multiplier bimonoids from $(A,t_1,t_2,e)$ to
$(A',t'_1,t'_2,e')$ is a morphism $f:A\nto A'$ in the category $\cm$ of
Proposition~\ref{prop:category-M} whose components render commutative the
following diagrams. 
\begin{equation}\label{eq:mbm_morphism}
\xymatrix{
A'A\ar[r]^-{f_2} \ar[rd]_-{e'e} &
A'\ar[d]^-{e'}
&&
A^{\prime 2}A^2 \ar[r]^-{1f_21} \ar[d]_-{t'_2t_1} &
A^{\prime 2}A \ar[r]^-{t'_21} &
A^{\prime 2}A \ar[d]^-{1f_2}\\
& I
&&
A^{\prime 2}A^2 \ar[r]_-{1c1} &
(A'A)^2 \ar[r]_-{f_2f_2} &
A^{\prime 2}}
\end{equation}
Once again, there is an equivalent, symmetric set of diagrams with the roles of
the components interchanged:
$$
\xymatrix{
AA'\ar[r]^-{f_1} \ar[rd]_-{ee'} &
A'\ar[d]^-{e'}
&&
A^2A^{\prime 2} \ar[r]^-{1f_11} \ar[d]_-{t_2t'_1} &
AA^{\prime 2} \ar[r]^-{1t'_1} &
AA^{\prime 2} \ar[d]^-{f_11}\\
& I
&&
A^2A^{\prime 2} \ar[r]_-{1c1} &
(AA')^2 \ar[r]_-{f_1f_1} &
A^{\prime 2}}
$$
Moreover, using the non-degeneracy of $A^{\prime 2}$, the second diagram of
\eqref{eq:mbm_morphism} is seen to be equivalent also to either of the
symmetric diagrams 
$$
\xymatrix@R=15pt{
A'A^2A' \ar[r]^-{1t_11} \ar[dddd]_-{f_2f_1} &
A'A^2A'\ar[r]^-{11f_1} &
A'AA' \ar[d]^-{c1}
&&
A'A^2A' \ar[r]^-{1t_21} \ar[dddd]_-{f_2f_1} &
A'A^2A'\ar[r]^-{f_211} &
A'AA' \ar[d]^-{1c} \\
&& AA^{\prime 2} \ar[d]^-{1t'_1}
&&
&& A^{\prime 2}A \ar[d]^-{t'_21} \\
&& AA^{\prime 2} \ar[d]^-{c^{-1}1}
&&
&& A^{\prime 2}A \ar[d]^-{1c^{-1}} \\
&& A'AA' \ar[d]^-{f_21}
&&
&& A'AA' \ar[d]^-{1f_1} \\
A^{\prime 2}\ar[rr]_-{t'_1} &&
A^{\prime 2}
&&
A^{\prime 2}\ar[rr]_-{t'_2} &&
A^{\prime 2} .}
$$

We only need to show that for the particular components $d_1$ and $d_2$ in
Theorem \ref{thm:mbm_vs_comonoid}~(i), commutativity of the second diagram in
\eqref{eq:comonoidal-morphism} becomes equivalent to commutativity of the
second diagram in \eqref{eq:mbm_morphism}. In terms of strings, this says that
the first and last composites below are equal; but since the first three are
always equal by \eqref{eq:multiplicative} and
\eqref{eq:component-compatibility} for $f$, this is equivalent to the last two
composites being equal. 
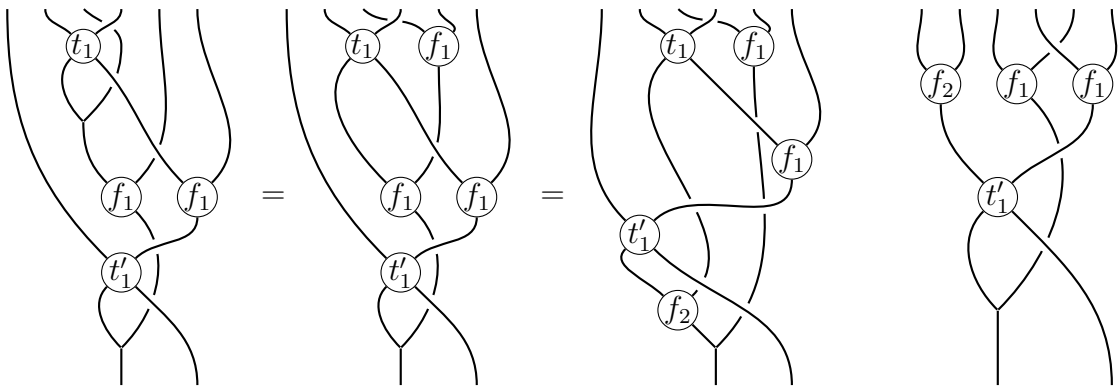
\begin{equation*}
  \begin{tikzpicture}
\path (1,3.5) node[empty,name=m] {}
(1,4.5) node[arr,name=t1u] {$t_1$}
(1.5,1.5) node[arr,name=t1d] {$t'_1$}
(1.5,2.5) node[arr,name=f1l] {$f_1$}
(2.5,2.5) node[arr,name=f1r] {$f_1$}
(1.5,4.5) node[empty,name=x] {} 
(1.5,0.5) node[empty,name=m1] {}
(0,5) node[empty,name=in1] {}
(0.5,5) node[empty,name=in2] {}
(1,5) node[empty,name=in3] {} 
(1.5,5) node[empty,name=in4] {}
(2,5) node[empty,name=in5] {}
(2.5,5) node[empty,name=in6] {}
(1.5,0) node[empty,name=b1] {}
(2.5,0) node[empty,name=b2] {};
\draw[braid] (in1) to[out=270,in=135] (t1d);
\draw[braid] (in2) to[out=270,in=135] (t1u);
\draw[braid, name path=bin3] (in3) to[out=315,in=90] (x) to[out=270,in=45] (m);
\path[braid, name path=bin4] (in4) to[out=270,in=45] (t1u);
\fill[white,name intersections={of=bin3 and bin4}] (intersection-1) circle(0.1);
\draw[braid] (in4) to[out=270,in=45] (t1u);
\draw[braid, name path=bin5] (in5) to[out=270,in=45] (f1l);
\draw[braid, name path=bin6] (in6) to[out=270,in=45] (f1r);
\draw[braid] (t1u) to[out=225,in=135] (m);
\path[braid, name path=t1uf1r] (t1u) to[out=315,in=135] (f1r);
\fill[white,name intersections={of=bin3 and t1uf1r}] (intersection-1) circle(0.1);
\fill[white,name intersections={of=bin5 and t1uf1r}] (intersection-1) circle(0.1);
\draw[braid] (t1u) to[out=315,in=135] (f1r);
\draw[braid] (m) to[out=270,in=135] (f1l);
\draw[braid, name path=f1lm1] (f1l) to[out=315,in=45] (m1);
\path[braid, name path=f1rt1d] (f1r) to[out=270,in=45] (t1d);
\fill[white,name intersections={of=f1lm1 and f1rt1d}] (intersection-1) circle(0.1);
\draw[braid] (f1r) to[out=270,in=45] (t1d);
\draw[braid] (t1d) to[out=225,in=135] (m1);
\path[braid, name path=t1db2] (t1d) to[out=315,in=90] (b2);
\fill[white,name intersections={of=t1db2 and f1lm1}] (intersection-1) circle(0.1);
\draw[braid] (t1d) to[out=315,in=90] (b2);
\draw[braid] (m1) to (b1);
\path (3.5,2.5) node[empty] {$=$};
  \end{tikzpicture}
  \begin{tikzpicture}
\path (1,3.5) node[empty,name=m] {}
(1,4.5) node[arr,name=t1u] {$t_1$}
(1.5,1.5) node[arr,name=t1d] {$t'_1$}
(1.5,2.5) node[arr,name=f1l] {$f_1$}
(2.5,2.5) node[arr,name=f1r] {$f_1$}
(2,4.5) node[arr,name=f1u] {$f_1$}
(1.5,4.5) node[empty,name=x] {} 
(1.5,0.5) node[empty,name=m1] {}
(0,5) node[empty,name=in1] {}
(0.5,5) node[empty,name=in2] {}
(1,5) node[empty,name=in3] {} 
(1.5,5) node[empty,name=in4] {}
(2,5) node[empty,name=in5] {}
(2.5,5) node[empty,name=in6] {}
(1.5,0) node[empty,name=b1] {}
(2.5,0) node[empty,name=b2] {};
\draw[braid] (in1) to[out=270,in=135] (t1d);
\draw[braid] (in2) to[out=270,in=135] (t1u);
\draw[braid, name path=bin3] (in3) to[out=315,in=135] (f1u);
\path[braid, name path=bin4] (in4) to[out=270,in=45] (t1u);
\fill[white,name intersections={of=bin3 and bin4}] (intersection-1) circle(0.1);
\draw[braid] (in4) to[out=270,in=45] (t1u);
\draw[braid, name path=bin5] (in5) to[out=270,in=45] (f1u);
\draw[braid, name path=bin6] (in6) to[out=270,in=45] (f1r);
\draw[braid] (t1u) to[out=225,in=135] (f1l);
 \path[braid, name path=t1uf1r] (t1u) to[out=315,in=135] (f1r);
 \draw[braid, name path=f1uf1l] (f1u) to[out=270,in=45] (f1l);
 \fill[white,name intersections={of=t1uf1r and f1uf1l}] (intersection-1) circle(0.1);
 \draw[braid] (t1u) to[out=315,in=135] (f1r);
\draw[braid, name path=f1lm1] (f1l) to[out=315,in=45] (m1);
\path[braid, name path=f1rt1d] (f1r) to[out=270,in=45] (t1d);
\fill[white,name intersections={of=f1lm1 and f1rt1d}] (intersection-1) circle(0.1);
\draw[braid] (f1r) to[out=270,in=45] (t1d);
\draw[braid] (t1d) to[out=225,in=135] (m1);
\path[braid, name path=t1db2] (t1d) to[out=315,in=90] (b2);
\fill[white,name intersections={of=t1db2 and f1lm1}] (intersection-1) circle(0.1);
\draw[braid] (t1d) to[out=315,in=90] (b2);
\draw[braid] (m1) to (b1);
\path (3.5,2.5) node[empty] {$=$};
  \end{tikzpicture}
  \begin{tikzpicture}
\path (1,3.5) node[empty,name=m] {}
(1,4.5) node[arr,name=t1u] {$t_1$}
(0.5,2.0) node[arr,name=t1d] {$t'_1$}
(1.0,1.0) node[arr,name=f1l] {$f_2$}
(2.5,3.0) node[arr,name=f1r] {$f_1$}
(2,4.5) node[arr,name=f1u] {$f_1$}
(1.5,4.5) node[empty,name=x] {} 
(1.5,0.5) node[empty,name=m1] {}
(0,5) node[empty,name=in1] {}
(0.5,5) node[empty,name=in2] {}
(1,5) node[empty,name=in3] {} 
(1.5,5) node[empty,name=in4] {}
(2,5) node[empty,name=in5] {}
(2.5,5) node[empty,name=in6] {}
(1.5,0) node[empty,name=b1] {}
(2.5,0) node[empty,name=b2] {};
\draw[braid] (in1) to[out=270,in=135] (t1d);
\draw[braid] (in2) to[out=270,in=135] (t1u);
\draw[braid, name path=bin3] (in3) to[out=315,in=135] (f1u);
\path[braid, name path=bin4] (in4) to[out=270,in=45] (t1u);
\fill[white,name intersections={of=bin3 and bin4}] (intersection-1) circle(0.1);
\draw[braid, name path=bin5] (in5) to[out=270,in=45] (f1u);
\draw[braid, name path=bin6] (in6) to[out=270,in=45] (f1r);
\draw[braid, name path=t1uf1l] (t1u) to[out=225,in=45] (f1l);
\draw[braid] (in4) to[out=270,in=45] (t1u);
\path[braid, name path=t1uf1r] (t1u) to[out=315,in=135] (f1r);
\draw[braid, name path=f1um1] (f1u) to[out=270,in=45] (m1);
\fill[white,name intersections={of=t1uf1r and f1um1}] (intersection-1) circle(0.1);
\draw[braid] (t1u) to[out=315,in=135] (f1r);
\draw[braid, name path=f1lm1] (f1l) to[out=315,in=135] (m1);
 \path[braid, name path=f1rt1d] (f1r) to[out=270,in=45] (t1d);
 \fill[white,name intersections={of=f1um1 and f1rt1d}] (intersection-1) circle(0.1);
\fill[white,name intersections={of=t1uf1l and f1rt1d}] (intersection-1) circle(0.1);
 \draw[braid] (f1r) to[out=270,in=45] (t1d);
\draw[braid] (t1d) to[out=225,in=135] (f1l);
\path[braid, name path=t1db2] (t1d) to[out=315,in=90] (b2);
\fill[white,name intersections={of=t1db2 and f1um1}] (intersection-1) circle(0.1);
\fill[white,name intersections={of=t1db2 and t1uf1l}] (intersection-1) circle(0.1);
\draw[braid] (t1d) to[out=315,in=90] (b2);
\draw[braid] (m1) to (b1);
  \end{tikzpicture}
\qquad
\begin{tikzpicture}
  \path (1,1.0) node[empty,name=m] {}
(1.0,2.5) node[arr,name=t1d] {$t'_1$}
(1.25,4.0) node[arr,name=f1l] {$f_1$}
(2.25,4.0) node[arr,name=f1r] {$f_1$}
(0.25,4) node[arr,name=f2] {$f_2$}
(1.5,4.5) node[empty,name=x] {} 
(0,5) node[empty,name=in1] {}
(0.5,5) node[empty,name=in2] {}
(1,5) node[empty,name=in3] {} 
(1.5,5) node[empty,name=in4] {}
(2,5) node[empty,name=in5] {}
(2.5,5) node[empty,name=in6] {}
(1.0,0) node[empty,name=b1] {}
(2.5,0) node[empty,name=b2] {};
\draw[braid] (in1) to[out=270,in=135] (f2);
\draw[braid] (in2) to[out=270,in=45] (f2);
\draw[braid, name path=bin3] (in3) to[out=270,in=135] (f1l);
\path[braid, name path=bin4] (in4) to[out=270,in=135] (f1r);
\draw[braid, name path=bin5] (in5) to[out=270,in=45] (f1l);
\fill[white,name intersections={of=bin4 and bin5}] (intersection-1) circle(0.1);
\draw[braid] (in4) to[out=270,in=135] (f1r);
\draw[braid] (in6) to[out=270,in=45] (f1r);
\draw[braid] (f2) to[out=270,in=135] (t1d);
\draw[braid, name path=f1lm] (f1l) to[out=315,in=45] (m);
\path[braid, name path=f1rt1d] (f1r) to[out=270,in=45] (t1d);
\fill[white,name intersections={of=f1lm and f1rt1d}] (intersection-1) circle(0.1);
\draw[braid] (f1r) to[out=270,in=45] (t1d);
\draw[braid] (t1d) to[out=225,in=135] (m);
\path[braid, name path=t1db2] (t1d) to[out=315,in=90] (b2);
\fill[white,name intersections={of=f1lm and t1db2}] (intersection-1) circle(0.1);
\draw[braid] (t1d) to[out=315,in=90] (b2);
\draw[braid] (m) to (b1);
\end{tikzpicture}
\end{equation*}
Since $f_1$ belongs to \cq and the multiplication is non-degenerate, this is
in turn equivalent to the equality of the following composites.   
\begin{equation*}
  \begin{tikzpicture}
\path 
(1,3) node[arr,name=t1u] {$t_1$}
(0.5,2.0) node[arr,name=t1d] {$t'_1$}
(1.0,1.0) node[arr,name=f1l] {$f_2$}
(1.5,2.5) node[arr,name=f1r] {$f_1$}
(0.5,0.5) node[empty,name=m] {}
(-0.5,4) node[empty,name=in0] {}
(0,4) node[empty,name=in1] {}
(0.5,4) node[empty,name=in2] {}
(1.5,4) node[empty,name=in4] {}
(2.0,4) node[empty,name=in6] {}
(0.5,0) node[empty,name=b1] {}
(2.0,0) node[empty,name=b2] {};
\draw[braid] (in0) to[out=270,in=135] (m);
\draw[braid] (in1) to[out=270,in=135] (t1d);
\draw[braid] (in2) to[out=270,in=135] (t1u);
\draw[braid, name path=bin6] (in6) to[out=270,in=45] (f1r);
\draw[braid, name path=t1uf1l] (t1u) to[out=225,in=45] (f1l);
\draw[braid] (in4) to[out=270,in=45] (t1u);
\path[braid, name path=t1uf1r] (t1u) to[out=315,in=135] (f1r);
\draw[braid] (t1u) to[out=315,in=135] (f1r);
\draw[braid, name path=f1lm1] (f1l) to[out=270,in=45] (m);
 \path[braid, name path=f1rt1d] (f1r) to[out=270,in=45] (t1d);
\fill[white,name intersections={of=t1uf1l and f1rt1d}] (intersection-1) circle(0.1);
 \draw[braid] (f1r) to[out=270,in=45] (t1d);
\draw[braid] (t1d) to[out=225,in=135] (f1l);
\path[braid, name path=t1db2] (t1d) to[out=315,in=45] (b2);
\fill[white,name intersections={of=t1db2 and t1uf1l}] (intersection-1) circle(0.1);
\draw[braid] (t1d) to[out=315,in=45] (b2);
\draw[braid] (m) to (b1);
  \end{tikzpicture}
\qquad
\begin{tikzpicture}
  \path 
(1.0,1.5) node[arr,name=t1d] {$t'_1$}
(1.75,3.0) node[arr,name=f1r] {$f_1$}
(0.25,3) node[arr,name=f2] {$f_2$}
(0.5,0.5) node[empty,name=m] {}
(-0.5,4) node[empty,name=in0] {}
(0,4) node[empty,name=in1] {}
(0.5,4) node[empty,name=in2] {}
(1.5,4) node[empty,name=in4] {}
(2.0,4) node[empty,name=in6] {}
(0.5,0) node[empty,name=b1] {}
(2.0,0) node[empty,name=b2] {};
\draw[braid] (in0) to[out=270,in=135] (m);
\draw[braid] (in1) to[out=270,in=135] (f2);
\draw[braid] (in2) to[out=270,in=45] (f2);
\path[braid, name path=bin4] (in4) to[out=270,in=135] (f1r);
\draw[braid] (in4) to[out=270,in=135] (f1r);
\draw[braid] (in6) to[out=270,in=45] (f1r);
\draw[braid] (f2) to[out=270,in=135] (t1d);
\path[braid, name path=f1rt1d] (f1r) to[out=270,in=45] (t1d);
\draw[braid] (f1r) to[out=270,in=45] (t1d);
\draw[braid] (t1d) to[out=225,in=45] (m);
\path[braid, name path=t1db2] (t1d) to[out=315,in=90] (b2);
\draw[braid] (t1d) to[out=315,in=90] (b2);
\draw[braid] (m) to (b1);
\end{tikzpicture}  
\end{equation*}
In the left diagram, use \eqref{eq:multiplier} for $f$,
\eqref{eq:t_1-2_short_compatibility}, and \eqref{eq:component-compatibility} 
for $f$; in the right, use \eqref{eq:t_1-2_short_compatibility} and
\eqref{eq:component-compatibility} for $f$. The equality of the resulting
composites is equivalent, by non-degeneracy, to commutativity of the second
diagram in \eqref{eq:mbm_morphism}. 

\begin{example}
Let $A$ and $A'$ be multiplier bimonoids satisfying the conditions in 
Theorem \ref{thm:mbm_vs_comonoid} (ii) and let $g\colon A\to A'$ be a
morphism in the category $\mathcal D$ of Proposition \ref{prop:f-sharp}. 
Then $g^\#$ is a morphism of multiplier bimonads if and only if  $e'.g=e$ and
$t'_1.gg=gg.t_1$.

Indeed, the top right path of the first diagram of \eqref{eq:mbm_morphism}
takes the form in any of the equal paths in
$$
\xymatrix{
A'A \ar[rr]^-{1g} \ar[d]_-{e'1} &&
A^{\prime 2} \ar[r]^-{m'}\ar[d]_-{e'e'} \ar@{}[rd]|-{\eqref{eq:e-multiplicative}}&
A' \ar[d]^-{e'} \\
A\ar[r]_-{g} &
A' \ar[r]_-{e'} & 
I \ar@{=}[r] &
I.}
$$
Since $e'1:A'A \to A$ is an epimorphism, this is equal to $e'e$ (in the left
bottom path of the first diagram of \eqref{eq:mbm_morphism}) if and only if
$e'.g=e$. 

The top right path of the second diagram of \eqref{eq:mbm_morphism} takes the
form of any of the equal paths in 
$$
\xymatrix{
A^{\prime 2}A^2 \ar[r]^-{11g1} \ar[d]_-{t'_211} &
A^{\prime 3}A \ar[rrr]^-{1m'1} \ar[d]_-{111g} &&&
A^{\prime 2}A \ar[r]^-{t'_21} &
A^{\prime 2}A \ar[d]^-{11g} \\
A^{\prime 2}A^2\ar[d]_-{11gg} &
A^{\prime 4}\ar[rrr]^-{1m'1} \ar[d]_-{11t'_1} 
\ar@{}[rrrdd]|-{\eqref{eq:short_fusion}}&&&
A^{\prime 3} \ar[r]^-{t'_21} \ar[dd]^-{1t'_1} 
\ar@{}[rdd]|-{\eqref{eq:t_1-2_short_compatibility}}&
A^{\prime 3} \ar[dddd]^-{1m'} \\
A^{\prime 4}\ar[dd]_-{11t'_1} &
A^{\prime 4}\ar[d]_-{11c^{-1}} \\
&
A^{\prime 4} \ar[r]^-{1t'_11} \ar[d]_-{t'_211} 
\ar@{}[rd]|-{\eqref{eq:t_1-2_short_compatibility}}&
A^{\prime 4} \ar[r]^-{11c} \ar[d]^-{m'11} &
A^{\prime 4} \ar[r]^-{1m'1} \ar[dd]_-{m'11} 
\ar@{}[rdd]|-{\quad \textrm{(associativity)}}&
A^{\prime 3}\ar[rdd]^-{m'1} &\\
A^{\prime 4} \ar[r]^-{11c^{-1}}\ar[d]_-{1c1} &
A^{\prime 4} \ar[r]^-{1m'1} &
A^{\prime 3}\ar[rd]^-{1c}\\
A^{\prime 4}\ar[rrr]_-{11m'} &&&
A^{\prime 3} \ar[rr]_-{m'1} &&
A^{\prime 2}.}
$$
Using the form of $d'_2$ together with the non-degeneracy and the
associativity of $m'$, this is equal to the composite $m'm'.1c1.11gg.t'_2t_1$
(occurring in the left bottom path of the first diagram of
\eqref{eq:mbm_morphism}) if and only if 
$$
m'm'.1c1.11t'_1.11gg.d'_211=m'm'.1c1.11gg.11t_1.d'_211. 
$$
Since $d'_211$ is an epimorphism and $C^{\prime 2}$ is non-degenerate, this is
equivalent to $t'_1.gg=gg.t_1$. 
\end{example}


\bibliographystyle{plain}

\end{document}